\documentclass[preprint,12pt]{elsarticle}

\usepackage{amssymb}

\usepackage{setspace}
\usepackage[utf8x]{inputenc}
\usepackage{amsmath,amsthm,amssymb,fontenc,color,textcomp,amsfonts,graphicx,graphics}
\usepackage[T1]{fontenc}
\usepackage{hyperref}
\usepackage[nottoc,notlot,notlof]{tocbibind}
\hypersetup{
    colorlinks=true,
    linkcolor=black,
    filecolor=blue,      
    urlcolor=blue,
    }
\usepackage{fancyhdr}
\usepackage{xcolor}
\usepackage{tikz}
\usepackage{float}
\usepackage{ragged2e}

\usepackage[Sonny]{fncychap}
\usepackage{listings}
\usepackage{xcolor} \definecolor{Darkgreen}{rgb}{0,0.4,0}
\usepackage{color}
\definecolor{listinggray}{gray}{0.9}
\definecolor{lbcolor}{rgb}{0.9,0.9,0.9}

\definecolor{codegreen}{rgb}{0,0.6,0}
\definecolor{codegray}{rgb}{0.5,0.5,0.5}
\definecolor{codepurple}{rgb}{0.58,0,0.82}
\definecolor{backcolour}{rgb}{0.95,0.95,0.92}
\lstdefinestyle{mystyle}{
    backgroundcolor=\color{backcolour},   
    commentstyle=\color{codegreen},
    keywordstyle=\color{magenta},
    numberstyle=\tiny\color{codegray},
    stringstyle=\color{codepurple},
    basicstyle=\ttfamily\footnotesize,
    breakatwhitespace=true,         
    breaklines=true,                 
    captionpos=b,                    
    keepspaces=true,                 
    numbers=left,                    
    numbersep=5pt,                  
    showspaces=false,                
    showstringspaces=false,
    showtabs=false,                  
    tabsize=4
}

\newtheorem{thm}{Theorem}[section]

\newtheorem{df}[thm]{Definition}
\newtheorem{remark}{Remark}

\newtheorem{acknowledgements}{Acknowledgements}
\journal{Japan J. Indust. Appl. Math.}

\begin{document}

\begin{frontmatter}



\title{4D Segmentation Algorithm with application to 3D+time Image Segmentation}


\author[inst1,inst2]{Markjoe Olunna UBA}\corref{cor1}
\cortext[cor1]{Corresponding author}
\ead{markjoeuba@gmail.com}

\affiliation[inst1]{organization={Department of Mathematics and Descriptive Geometry, Slovak University of Technology in Bratislava
},
            addressline={Radlinsk\'{e}ho 11}, 
            city={Bratislava},
            postcode={810 05}, 
            country={Slovakia}}

\affiliation[inst2]{organization={Department of Mathematics, University of Nigeria
},
            addressline={Nsukka - Onitsha Rd}, 
            city={Nsukka},
            postcode={410 001}, 
            country={Nigeria}}
\author[inst1]{Karol MIKULA}
\author[inst1]{Seol Ah PARK}


\begin{abstract}
In this paper, we introduce and study a novel segmentation method for 4D images based on surface evolution governed by a nonlinear partial differential equation, the generalized subjective surface equation. 
The new method uses 4D digital image information and information from a thresholded 4D image in a local neighborhood. Thus, the 4D image segmentation is accomplished by defining the edge detector function's input as the weighted sum of the norm of gradients of presmoothed 4D image and norm of presmoothed thresholded 4D image in a local neighborhood. Additionally, we design and study a numerical method based on the finite volume approach for solving the new model. The reduced diamond cell approach is used for approximating the gradient of the solution. 
We use a semi-implicit finite volume scheme for the numerical discretization and show that our numerical scheme is unconditionally stable. The new 4D method was tested on artificial data and applied to real data representing 3D+time microscopy images of cell nuclei within the zebrafish pectoral fin and hind-brain. In a real application, processing 3D+time microscopy images amounts to solving a linear system with several billion unknowns and requires over $1000$ GB of memory; thus, it may not be possible to process these images on a serial machine without parallel implementation utilizing the MPI. Consequently, we develop and present in the paper OpenMP and MPI parallel implementation of designed algorithms. Finally, we include cell tracking results to show how our new method serves as a basis for finding trajectories of cells during embryogenesis.
\end{abstract}




\end{frontmatter}


\section{Introduction}
\label{sec:introduction}
\noindent In image processing and computer vision, it is well known that image segmentation is one of the fundamental and most studied problems. Consequently, there are several approaches to image segmentation in literature \cite{segm5, segm2, segm3, segm4, segm1, segm7, segm8, sethianb, segm6}. However, in this paper, we will
study and use the level set method for image segmentation. The idea of following the interfaces in immiscible fluid flow by the Eulerian approach was proposed by Dervieux and Thomasset \cite{levelset1, levelset2}, and Osher and Sethian \cite{oshersethian} then introduced the general level set methodology. The level set method is the
evolution of curves and surfaces by means of a dynamical embedding function (often referred to as the level set function). In other words, the level set method is the implicit representation of boundaries,
curves (or surfaces). Furthermore, in the level set formulation, the subjective surface (SUBSURF) method for image segmentation
will be the focus of this paper. This method, in the context of 
image processing, was introduced in \cite{sartietal, subsurf}, studied 
and applied in several biomedical research \cite{cmssga, kmnapmr, kmikula, app2, sartietal, subsurf}. The following works, \cite{baf, Dyballaetal, faure, parketal}, represent recent developments and
applications of the SUBSURF method. This method is used in Bioemergences workflow \cite{faure} (http://bioemergences.iscpif.fr/workflow/).
SUBSURF segmentation method is based on the idea of evolution of segmentation function governed by a nonlinear diffusion equation \cite{cmssga, sartietal, subsurf, zcrmbmpa}, which can be understood as an advection-diffusion model.
Hence, a segmentation seed (the starting point that determines
the approximate position of an object in the image) is usually needed to segment an object. Then an initial
segmentation function $u^0 (x)$ is constructed with reference to the segmentation seed. 
Finally, this segmentation function is allowed to evolve to the final state following the SUBSURF model. Ideally, the evolution process ends up with a function whose isosurfaces all have the object's shape
that is intended to be segmented.\\

\noindent In this paper, we introduce a generalization of the classical SUBSURF model. This generalization is due to the fact that in real applications where the object intended to be segmented possesses internal structures or edges, it is usually challenging to obtain optimal results using the classical SUBSURF segmentation approach. The reason is that this approach works with edge information throughout the segmentation process. Hence, edges within the internal structures in an object of interest are also respected during segmentation. To overcome the effect of the internal structures or edges, we introduced thresholding of image intensity values within a ball of appropriate radius around the object center. This local thresholding serves to eliminate the internal structures or edges. Additionally, we combined the information obtained from thresholding and original image intensities to get an accurate final segmentation result. Unlike the works in \cite{kmikula, msarti, app2}, on the numerical approximation of the subjective surface method, we utilized the rescaling of values of the segmentation function at each segmentation step to the interval [0,1], which makes the final decision on segmentation contour much easier and automatic. The computation task involving 4D microscopy images requires solving a linear system with several billion unknowns. This 4D high-scale computing problem requires over 1000 GB of memory and may not be possible to be solved on a serial machine without parallel implementation utilizing the MPI. Thus, we developed the parallelization of the 4D numerical method, which was
not done anywhere before, and showed its properties like unconditional stability and experimental order of convergence. Finally, we developed an efficient parallel implementation of the method for the generalized model in the case of 4D images.\\

\noindent The remaining part of this work is organized as follows: Section $2$ provides an overview of the SUBSURF method and some numerical
examples of the application of the SUBSURF segmentation method to artificial and real data.  Section $3$ contains the details of the proposed generalization of the classical method. It contains the studied model combining thresholded image information and original image data, the numerical scheme for solving the model, the stability of the numerical scheme, experimental order of convergence, an overview of the computer implementation, and parallel implementation using OpenMP and MPI.
Furthermore, the application of the new segmentation algorithm to $4$D image segmentation was presented in Section $4$. Finally, cell tracking based on the result of the new segmentation algorithm was presented in Section $5$.

\section{The SUBSURF segmentation method}\label{subsurfsection}
\noindent The SUBSURF segmentation method is based on the idea of evolution of segmentation function, which is governed by an advection-diffusion model. The SUBSURF model was proposed by Sarti, Malladi, and  Sethian \cite{subsurf} and is given by
\begin{eqnarray}\label{firsts}
  \frac{\partial u}{\partial t} = |\triangledown u|\triangledown \cdot \Big(g^0 \frac{\triangledown u}{|\triangledown u|} \Big) \text{ in } (0,T] \times \Omega,\\ \nonumber
 \end{eqnarray}
 where $u(t, x)$ is the unknown segmentation function, $g^0 = g(|\triangledown G_{\sigma}\ast I^0|)$, $I^0$ is image to be segmented, and $G_{\sigma}$ is a smoothing kernel. $g : \mathbb{R}_0^+ \rightarrow \mathbb{R^+}$ is a smooth nonincreasing function such that $g(0)=1$ and $0<g(s)\rightarrow 0$ as $s\rightarrow \infty$. The function $g$ was introduced by Perona and Malik in
\cite{pm} and has the following typical example:
\begin{eqnarray}\label{eg1}
 g(s)=\frac{1}{1+Ks^2}, ~ K\geq 0.
\end{eqnarray}
\noindent Equation (\ref{firsts}) is accompanied by
Dirichlet boundary conditions
\begin{eqnarray}\label{first1}
 u(t,x) = u^D \text{ on } [0,T] \times \partial \Omega,
\end{eqnarray}
or Neumann boundary conditions
\begin{eqnarray}\label{first2}
 \frac{\partial u}{\partial {n}}(t,x) = 0 \text{ on } [0,T] \times \partial \Omega,
\end{eqnarray}
where ${n}$ is unit normal to $\partial \Omega$, and with the initial condition
\begin{eqnarray}\label{first3}
 u(0,x) = u^0(x) \text{ in } \Omega.
\end{eqnarray}
Dirichlet boundary condition is used in the image segmentation, and without loss of generality $u^D = 0$ may be assumed. The zero Neumann boundary
conditions are used, e.g., in morphological image smoothing (see,
e.g., \cite{cmssga} and references therein).\\

\noindent In computations involving model {\rm(\ref{firsts})}, one possible choice of initial segmentation function, $u^0$, is the following function 
 \begin{eqnarray}\label{initialcond}
  u^0(x)=\frac{1}{|s-x|+v} ~,
 \end{eqnarray}
where $s$ corresponds to the focus point, and $\frac{1}{v}$ gives a maximum of $u^0$ whose peak is centered in a ``focus point'' inside the segmented object \cite{app2}. This function can also be considered only at a circle with center $s$ and radius $R$. Outside this circle, the value of $u^0$ is taken to be $\frac{1}{R+v}$. If {\rm(\ref{firsts})} is accompanied by zero Dirichlet boundary condition, then finally $\frac{1}{R+v}$ is subtracted from such peak-like profile. In the case of small objects, a smaller choice of value for $R$ can be used to speed up computations. Also, $u^0(x) = 1 - \frac{|x-s|}{R}$ is another possible choice of initial segmentation function.\\

\noindent  We now present some numerical and illustrative examples to demonstrate the application of the SUBSURF model to image segmentation. The model (\ref{firsts}) for surface reconstruction was tested on artificial
examples and applied to real data 
representing 3D microscopy images of cell nuclei. For the first numerical experiment, a sphere and a sphere with holes were generated, as can be seen in the first column of Figure \ref{figure717}. Afterward, model (\ref{firsts}) was used to reconstruct the shapes of these spheres, and the results obtained after reconstruction are shown in the second column of Figure \ref{figure717}. In the second experiment, the model (\ref{firsts}) was applied to real data representing 3D microscopy images of cell nuclei and membrane. In the first column of Figures \ref{figure29}, \ref{figure911}, and \ref{figure922}, 3D image of these microscopy images are shown, whereas, in the second column of Figures \ref{figure29}, \ref{figure911}, and \ref{figure922} the results obtained after application of model (\ref{firsts}) are shown and colored in black.
\begin{figure}[H]
	\centering
	\begin{minipage}{0.24\textwidth}
		\centering
		\includegraphics[width=\textwidth]{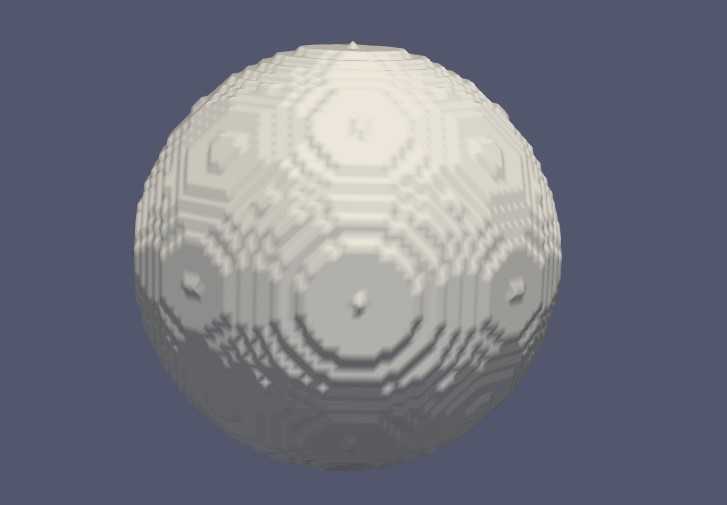}
	\end{minipage}
	\vspace{0.045cm}
	\begin{minipage}{0.24\textwidth}
		\centering
		\includegraphics[width=\textwidth]{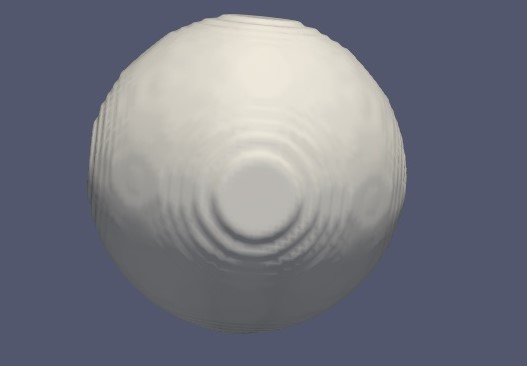}
	\end{minipage}
	\vspace{0.045cm}
	
	\begin{minipage}{0.24\textwidth}
		\centering
		\includegraphics[width=\textwidth]{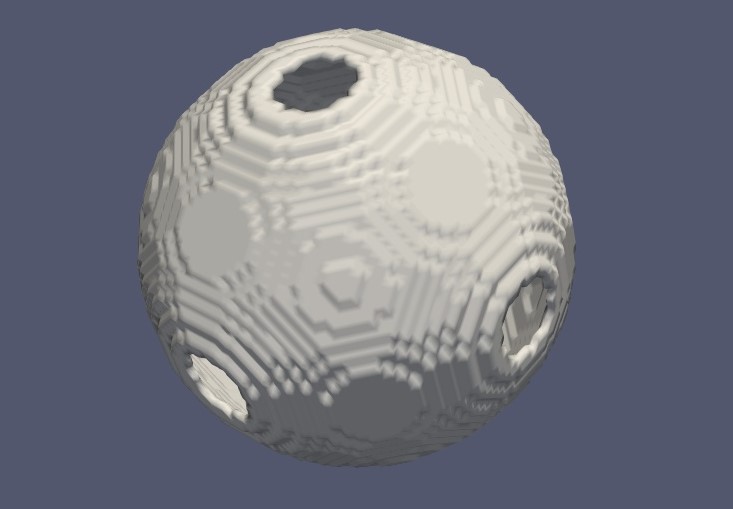}
	\end{minipage}
	\vspace{0.045cm}
	\begin{minipage}{0.24\textwidth}
		\centering
		\includegraphics[width=\textwidth]{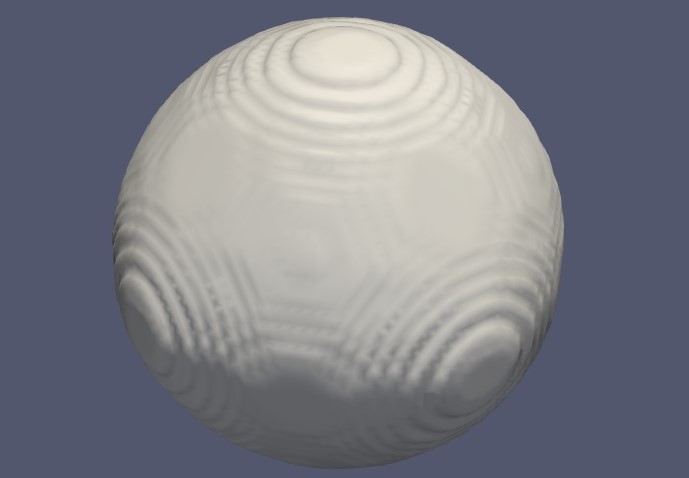}
	\end{minipage}
	\vspace{0.045cm}
	
	\caption{First column shows a solid sphere and sphere with six symmetric holes, whereas the second column shows results of the segmentation after the application of the model (\ref{firsts}).}
	\label{figure717}
\end{figure}

\begin{figure}[H]
\centering
	\begin{minipage}{0.24\textwidth}
		\centering
		\includegraphics[width=\textwidth]{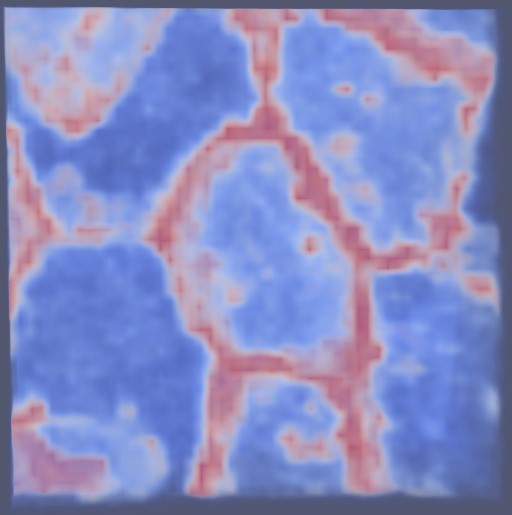}
	\end{minipage}
	\vspace{0.045cm}
	\begin{minipage}{0.24\textwidth}
		\centering
		\includegraphics[width=\textwidth]{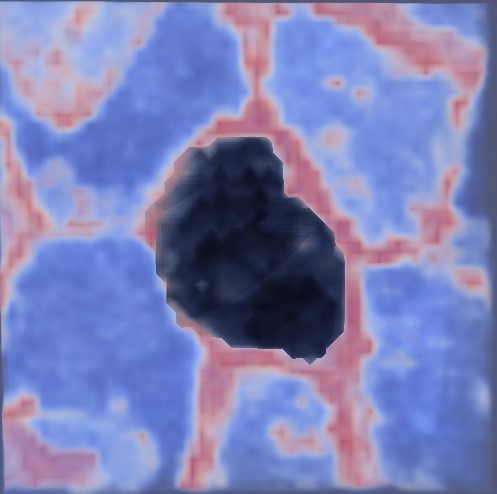}
	\end{minipage}
	\vspace{0.045cm}

	\caption{First column of this figure shows the 3D microscopy image of cell membrane to be segmented, whereas the second column shows the result after application of the model (\ref{firsts}); the cell membrane of interest and its corresponding result after segmentation (which are depicted in black) is located approximately at the center of each picture.}
	\label{figure29}
\end{figure}

 \begin{figure}[H]
 \centering
	\begin{minipage}{0.24\textwidth}
		\centering
		\includegraphics[width=\textwidth]{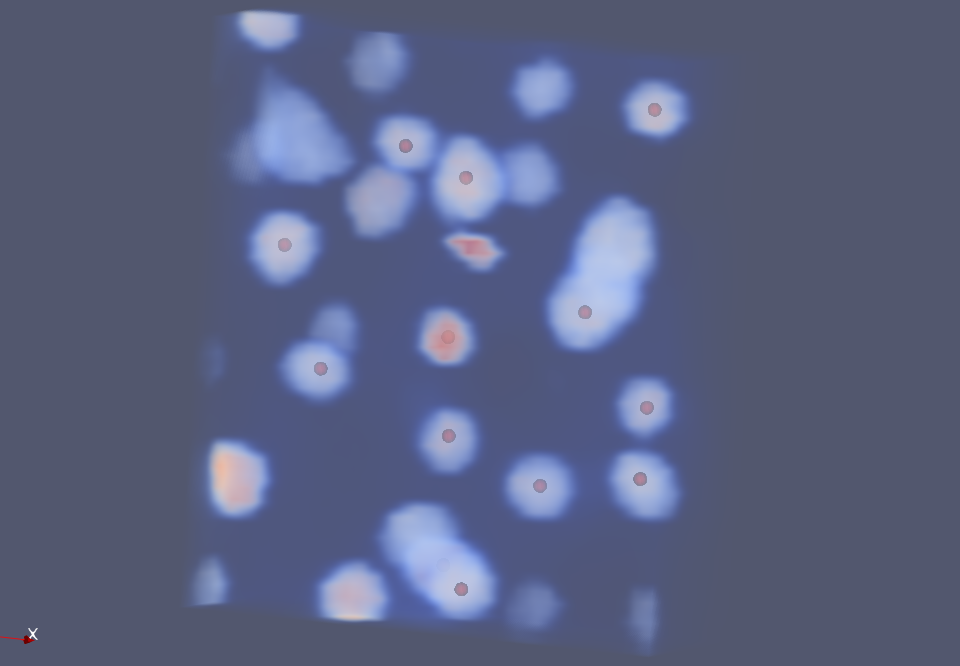}
	\end{minipage}
	\vspace{0.035cm}
	\begin{minipage}{0.24\textwidth}
		\centering
		\includegraphics[width=\textwidth]{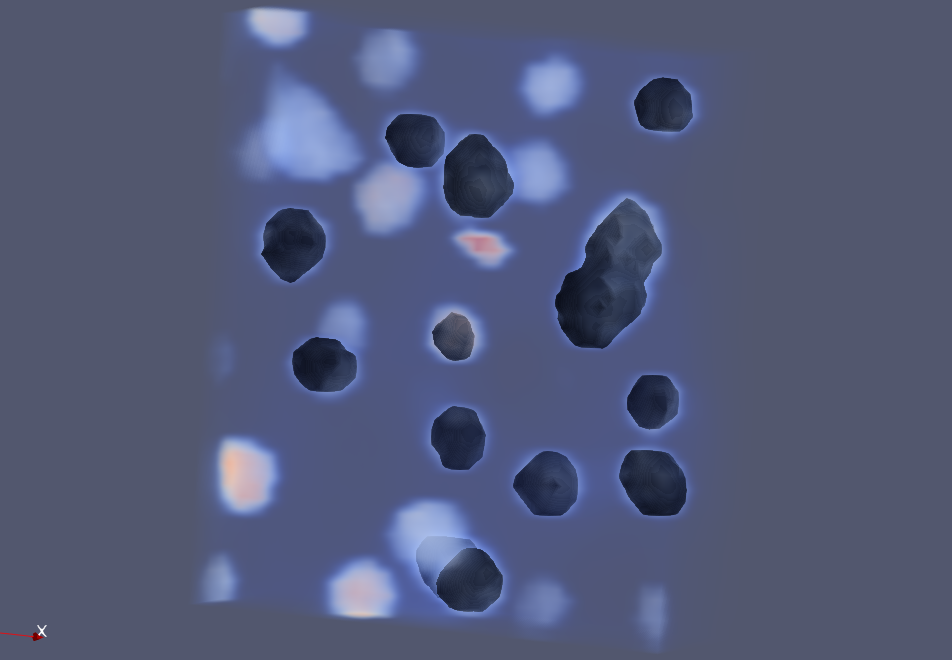}
	\end{minipage}
	\vspace{0.035cm}
	
	\caption{First column of this figure shows the 3D microscopy image of cell nuclei together with approximate cell centers (these are points marked in red color), whereas the second column shows the 3D image of microscopy image of cell nuclei after application of (\ref{firsts}) to the cells with centers marked in red color.}
	\label{figure911}
\end{figure}
\begin{figure}[H]
\centering
	\begin{minipage}{0.24\textwidth}
		\centering
		\includegraphics[width=\textwidth]{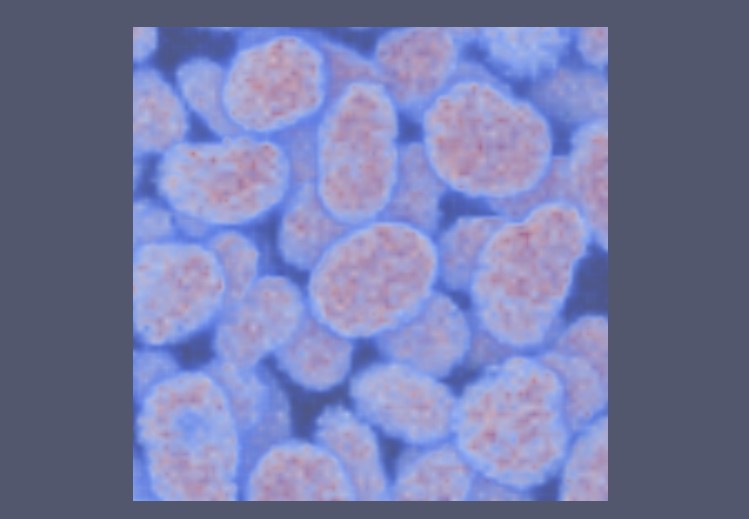}
	\end{minipage}
	\vspace{0.035cm}
	\begin{minipage}{0.24\textwidth}
		\centering
		\includegraphics[width=\textwidth]{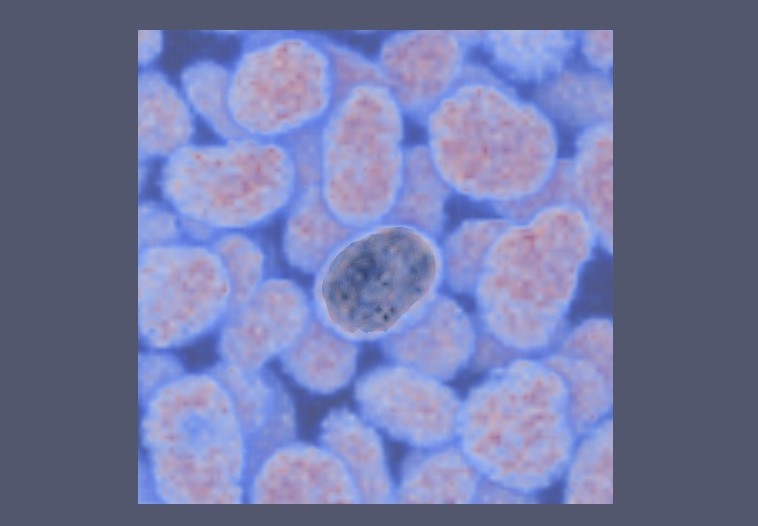}
	\end{minipage}
	\vspace{0.035cm}
	
	\begin{minipage}{0.24\textwidth}
		\centering
		\includegraphics[width=\textwidth]{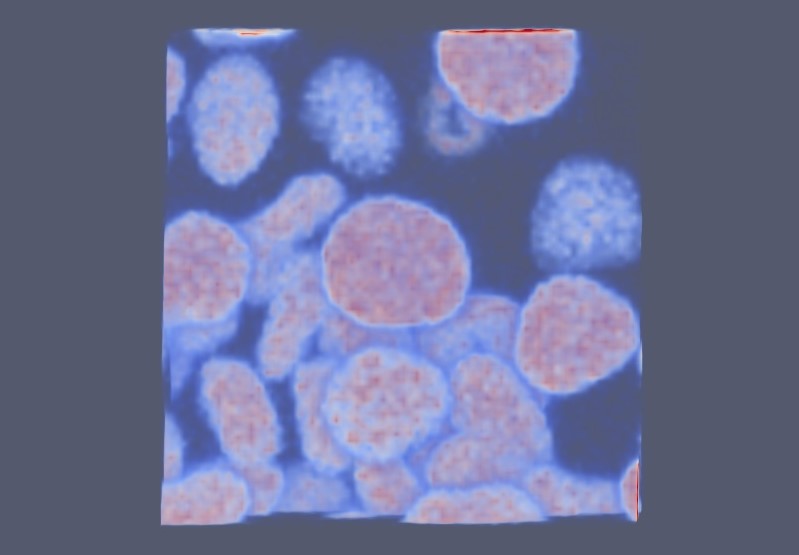}
	\end{minipage}
	\vspace{0.035cm}
	\begin{minipage}{0.24\textwidth}
		\centering
		\includegraphics[width=\textwidth]{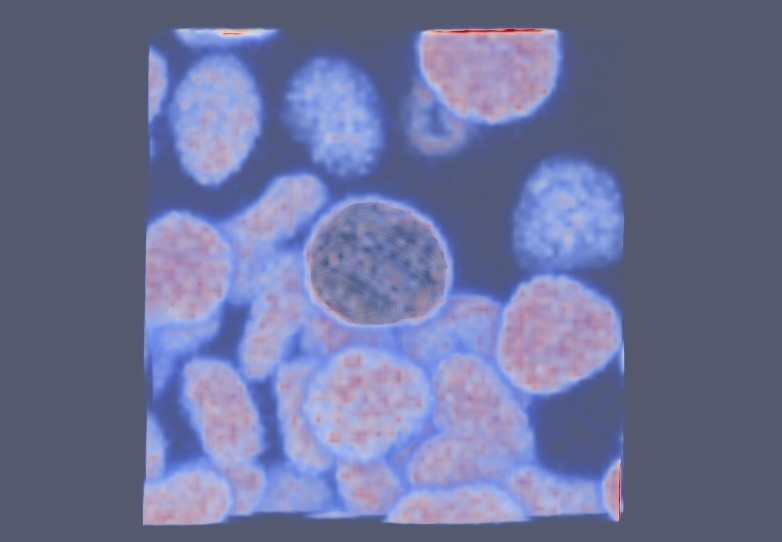}
	\end{minipage}
	\vspace{0.035cm}
	
	\caption{First column of this figure shows the 3D microscopy images of cell nuclei to be segmented, whereas the second column shows the result after application of the model (\ref{firsts}); the cell nuclei of interest and their corresponding results after segmentation (which are depicted in black) are located approximately at the center of each picture.}
	\label{figure922}
\end{figure}

\clearpage
\section{4D Image Segmentation Algorithm}\label{mainresults}
\noindent In the previous section, we presented the SUBSURF segmentation method and some of its successful applications.  
However, in many real applications where the object to be segmented has internal structures or edges, it is usually 
challenging to obtain optimal results using the SUBSURF segmentation approach (see, e.g., Figures \ref{mouse_figure222}, \ref{figure222}, and \ref{figure444}). The reason is that
this approach works with edge information throughout the segmentation process. Hence, edges within the internal structures in
an object of interest are also respected during segmentation. For instance, Figures \ref{mouse_figure222}, \ref{figure222}, and \ref{figure444} show the results of segmentation using the classical SUBSURF approach for $3$D images of zebrafish cell nuclei, cell membrane, and mouse embryo cell nuclei, respectively. The results of segmentation in these figures are colored in black, and the shapes of interest are located approximately at the center of each image.

\begin{figure}[H]
\centering
	\begin{minipage}{0.34\textwidth}
		\centering
		\includegraphics[width=\textwidth]{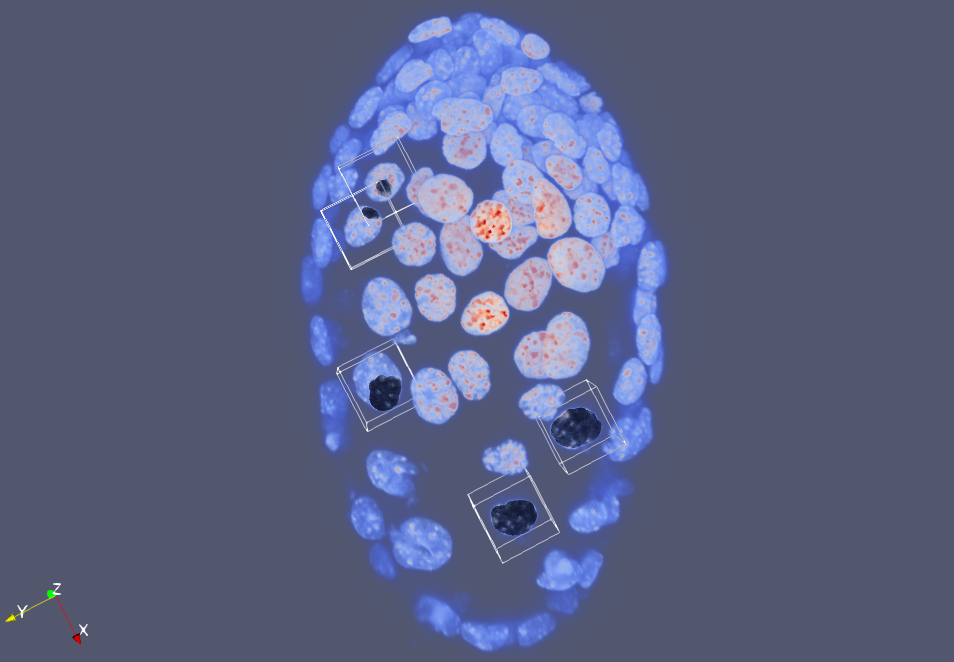}
	\end{minipage}
	\vspace{0.045cm}
	
	\caption{This figure shows the $3$D microscopy image of mouse embryo cell nuclei segmentation result using classical SUBSURF approach.}
	\label{mouse_figure222}
\end{figure}

\begin{figure}[H]
\centering
\begin{minipage}{0.24\textwidth}
		\centering
		\includegraphics[width=\textwidth]{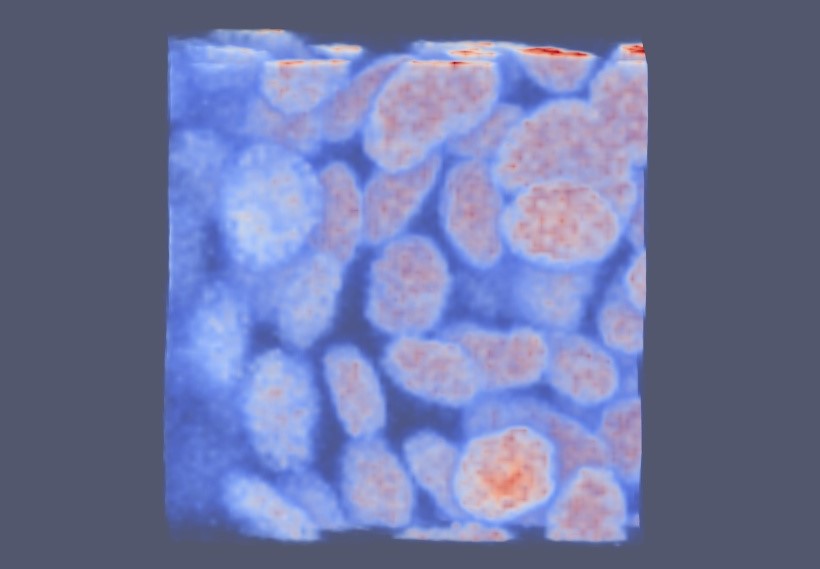}
	\end{minipage}
	\vspace{0.035cm}
	\begin{minipage}{0.24\textwidth}
		\centering
		\includegraphics[width=\textwidth]{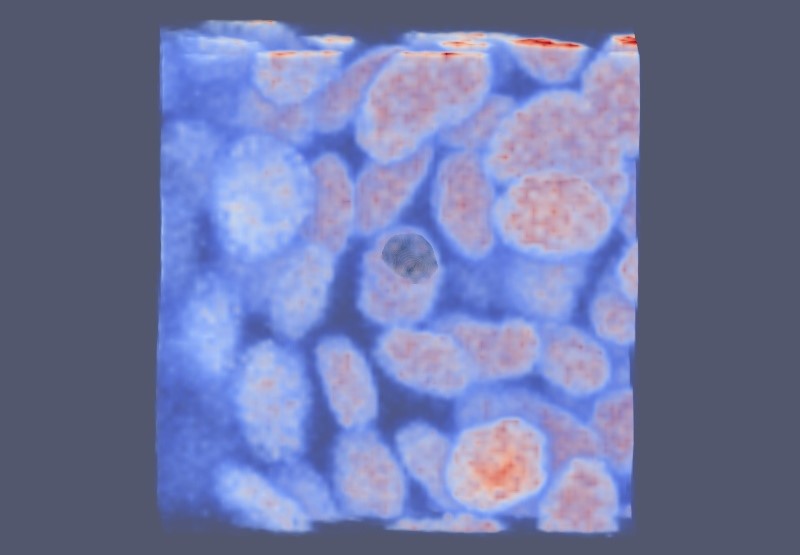}
	\end{minipage}
	\vspace{0.035cm}
	
	\begin{minipage}{0.24\textwidth}
		\centering
		\includegraphics[width=\textwidth]{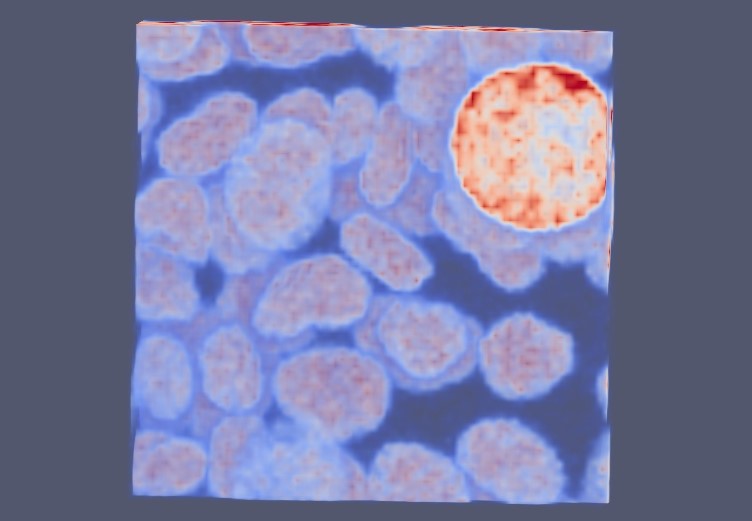}
	\end{minipage}
	\vspace{0.035cm}
	\begin{minipage}{0.24\textwidth}
		\centering
		\includegraphics[width=\textwidth]{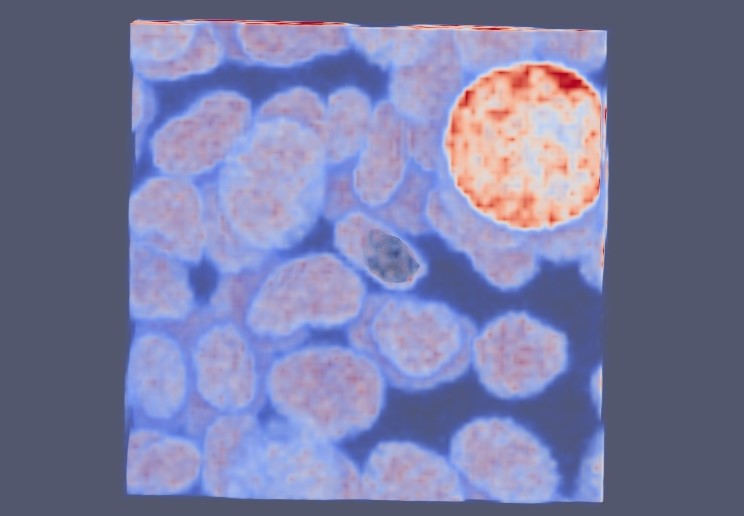}
	\end{minipage}
	\vspace{0.035cm}
	
	\caption{First column of this figure shows the 3D microscopy images of cell nuclei to be segmented, whereas the second column shows 
	the result of segmentation using the classical SUBSURF approach.}
	\label{figure222}
\end{figure}

\begin{figure}[H]
\centering
	\begin{minipage}{0.24\textwidth}
		\centering
		\includegraphics[width=\textwidth]{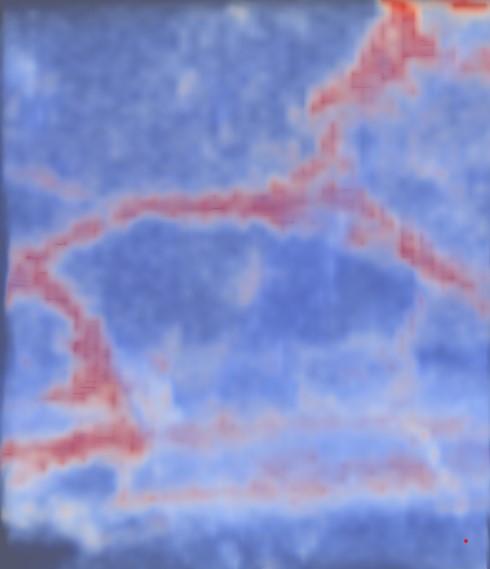}
	\end{minipage}
	\vspace{0.045cm}
	\begin{minipage}{0.24\textwidth}
		\centering
		\includegraphics[width=\textwidth]{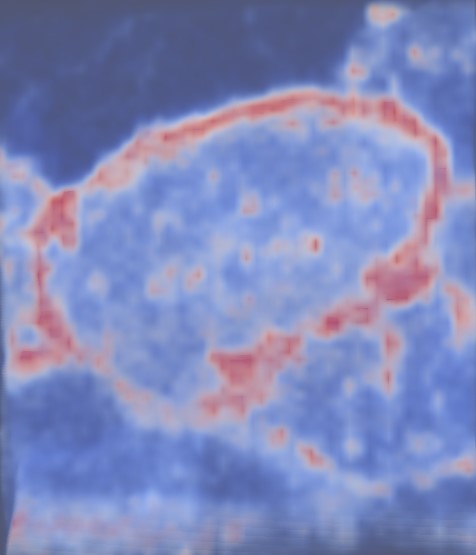}
	\end{minipage}
	\vspace{0.045cm}
	\begin{minipage}{0.27\textwidth}
		\centering
		\includegraphics[width=\textwidth]{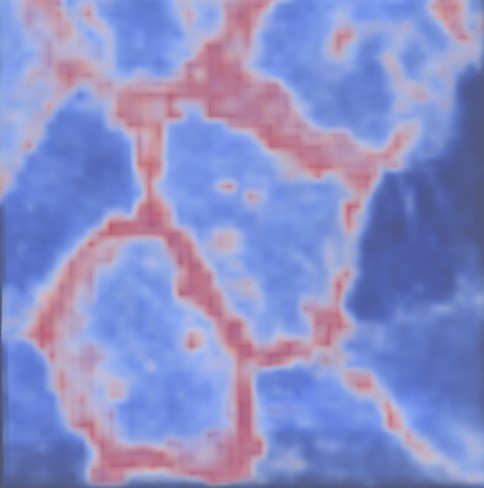}
	\end{minipage}
	\vspace{0.045cm}
	
	\begin{minipage}{0.24\textwidth}
		\centering
		\includegraphics[width=\textwidth]{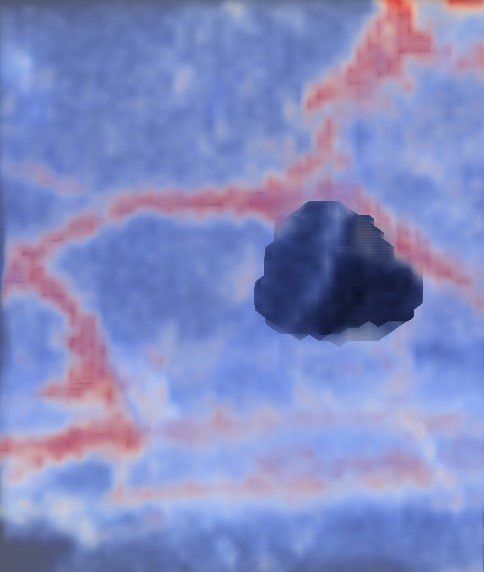}
	\end{minipage}
	\vspace{0.045cm}
	\begin{minipage}{0.24\textwidth}
		\centering
		\includegraphics[width=\textwidth]{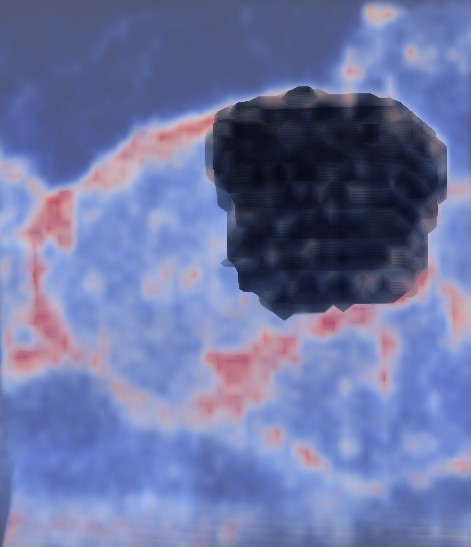}
	\end{minipage}
	\vspace{0.045cm}
	\begin{minipage}{0.27\textwidth}
		\centering
		\includegraphics[width=\textwidth]{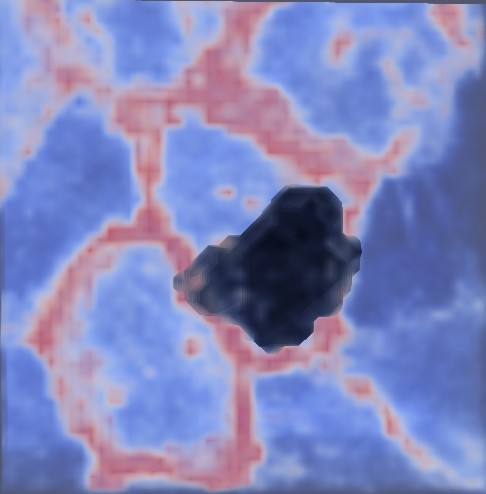}
	\end{minipage}
	\vspace{0.045cm}
	
	\caption{First row of this figure shows the 3D microscopy images of cell membranes to be segmented, while the second row shows 
	the result of segmentation using the classical SUBSURF approach.}
	\label{figure444}
\end{figure}

\noindent In this section, to overcome the effect of the internal structures or edges, we 
introduce the thresholding of values within a ball of appropriate radius around the object center. This local thresholding serves to
eliminate the internal structures or edges. Finally, we combine the information given by thresholding and original image intensities
to get a segmentation result. For example, in $3$D segmentation (see also \cite{ubaetal, ubaetalm}) of nuclei and membranes images shown in Figures \ref{mouse_figure222}--\ref{figure444} using this approach, we obtained the results presented in Figures \ref{mouse_figure1}--\ref{figure5}, which is clearly much more accurate compared to the results, shown in Figures \ref{mouse_figure222}--\ref{figure444}, obtained using the classical method.

\begin{figure}[H]
\centering
	\begin{minipage}{0.44\textwidth}
		\centering
		\includegraphics[width=\textwidth]{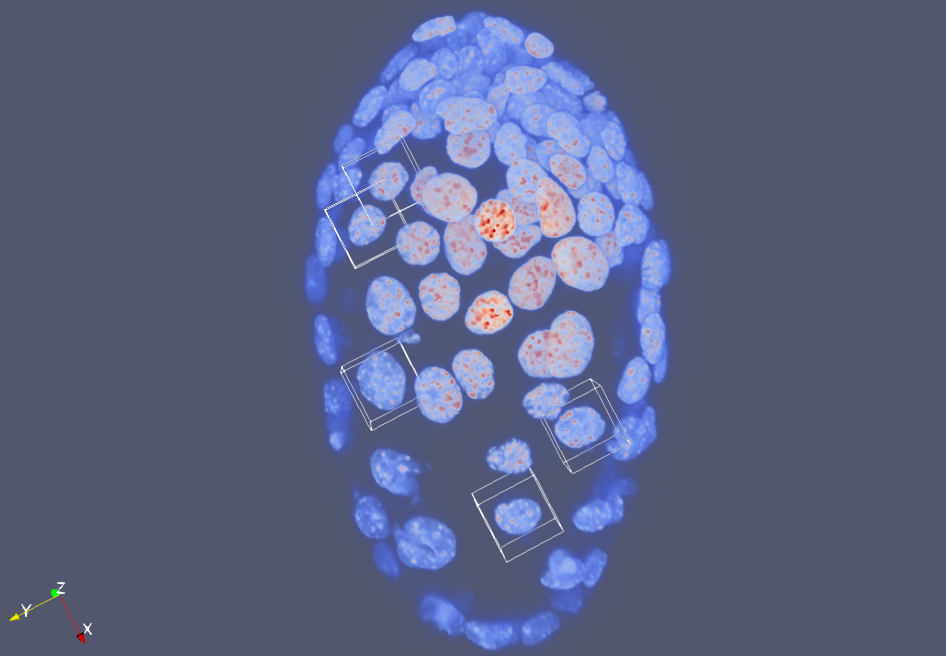}
	\end{minipage}
	\vspace{0.045cm}
	\begin{minipage}{0.44\textwidth}
		\centering
		\includegraphics[width=\textwidth]{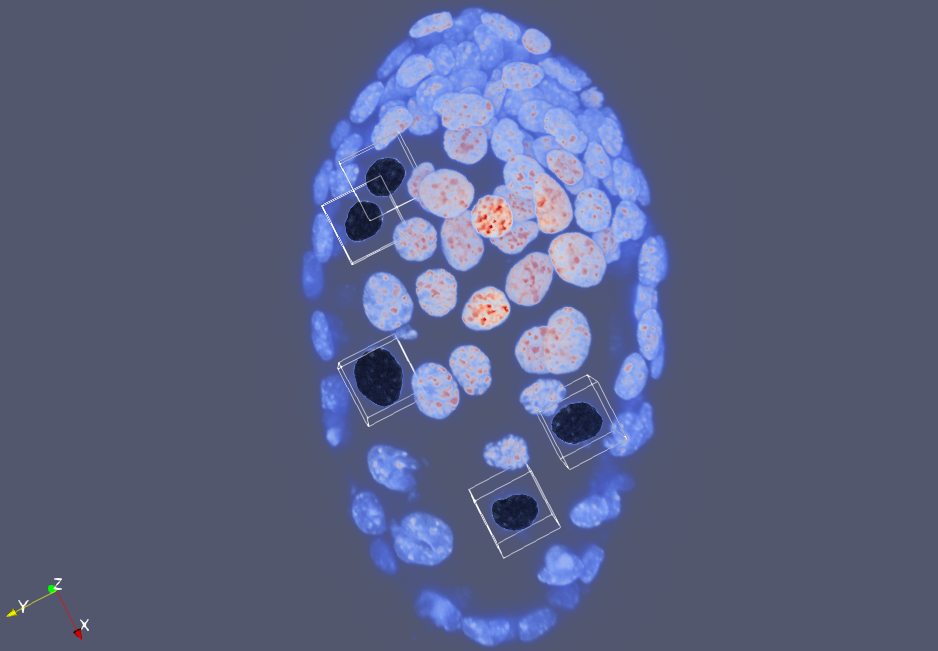}
	\end{minipage}
	\vspace{0.045cm}
	
	\caption{This figure shows the $3$D microscopy image of mouse embryo cell nuclei.}
	\label{mouse_figure1}
\end{figure}
\begin{figure}[H]
\centering
	\begin{minipage}{0.24\textwidth}
		\centering
		\includegraphics[width=\textwidth]{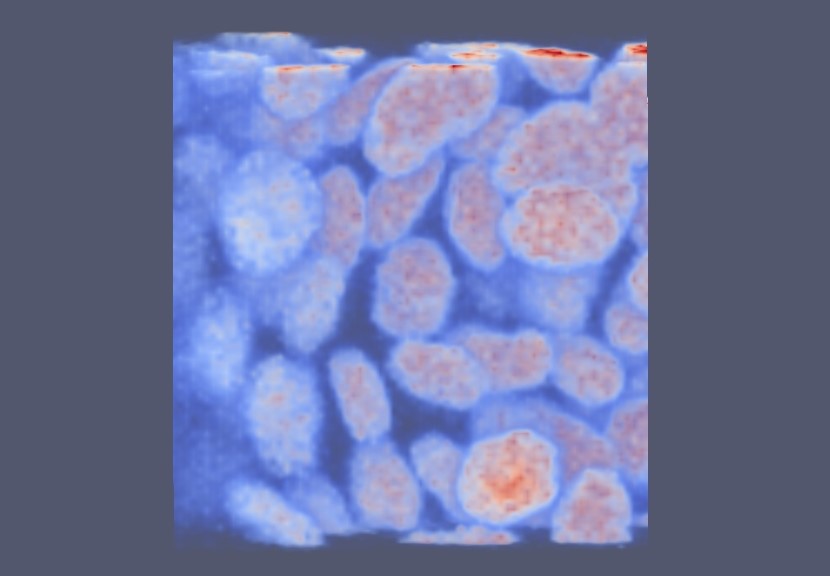}
	\end{minipage}
	\begin{minipage}{0.24\textwidth}
		\centering
		\includegraphics[width=\textwidth]{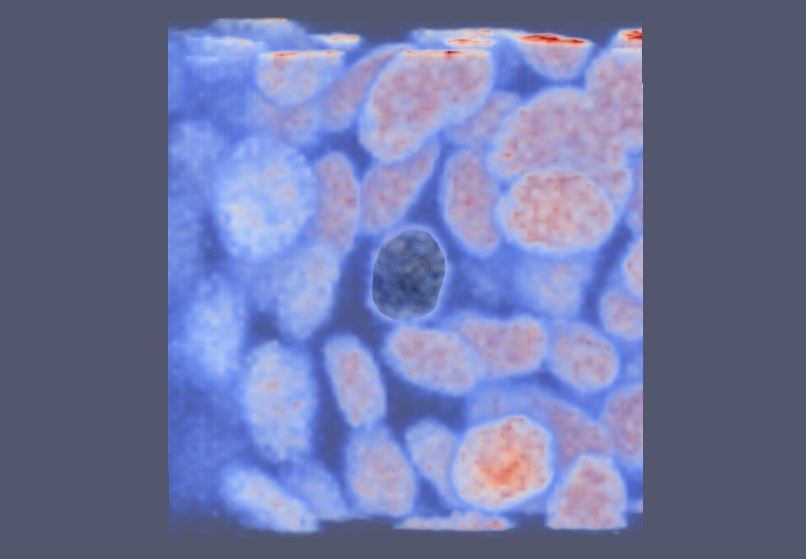}
	\end{minipage}
	\vspace{0.045cm}
	
	\vspace{0.045cm}
	\begin{minipage}{0.24\textwidth}
		\centering
		\includegraphics[width=\textwidth]{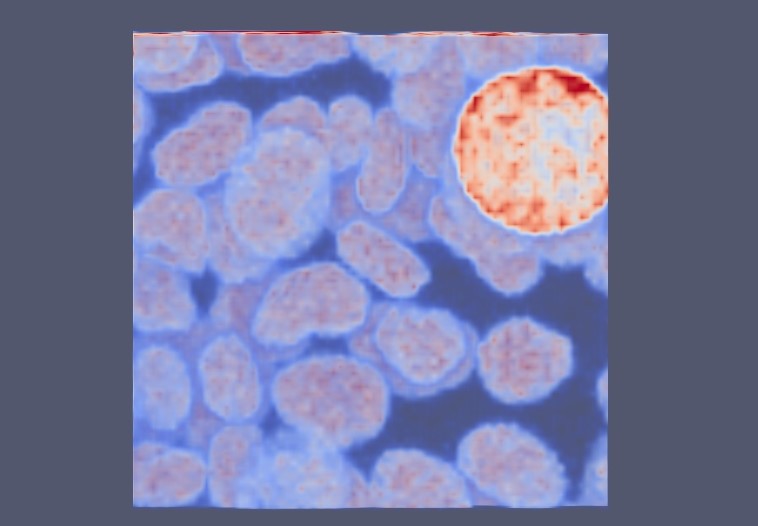}
	\end{minipage}
	\vspace{0.045cm}
	\begin{minipage}{0.24\textwidth}
		\centering
		\includegraphics[width=\textwidth]{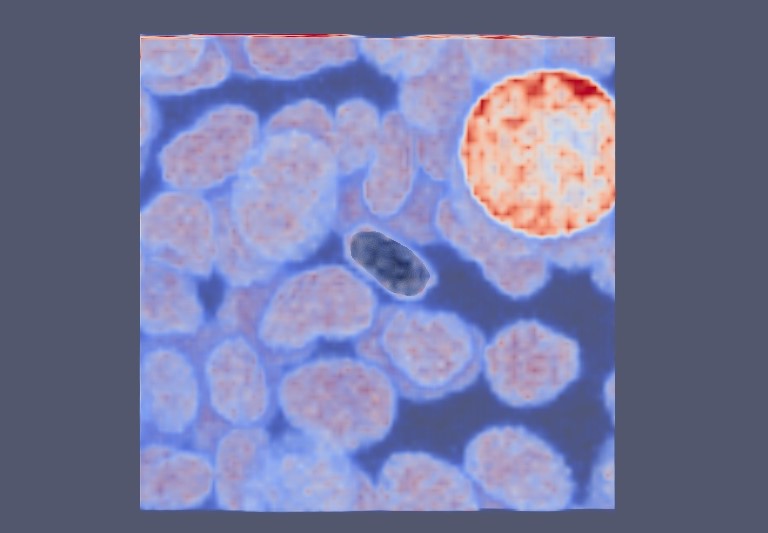}
	\end{minipage}
	\vspace{0.045cm}

	\caption{In this figure, first row shows the $3$D microscopy images of four different cell nuclei; second row shows
	segmentation result using thresholded image intensity information
	and original image intensity information in 3D case of (\ref{first}). That is, the result after  application of the 3D case of (\ref{first}) with $\delta = 0.5$  and $\vartheta = 0.5$.}
	\label{figure44}
\end{figure}
\begin{figure}[H]
\centering
	\begin{minipage}{0.2\textwidth}
		\centering
		\includegraphics[width=\textwidth]{membrane1tobesegmented.jpg}
	\end{minipage}
	\vspace{0.045cm}
	\begin{minipage}{0.2\textwidth}
		\centering
		\includegraphics[width=\textwidth]{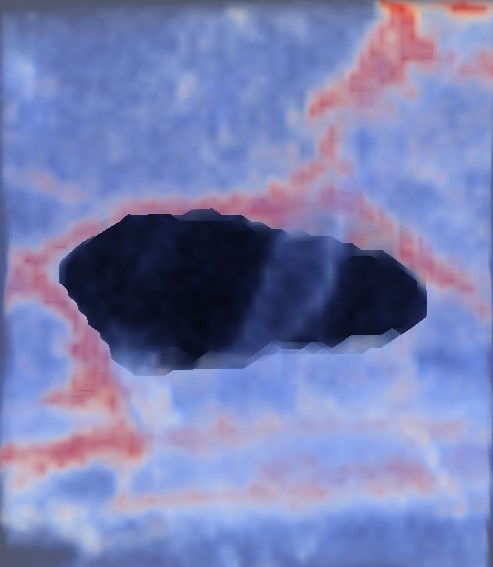}
	\end{minipage}
	\vspace{0.045cm}
	\begin{minipage}{0.2\textwidth}
		\centering
		\includegraphics[width=\textwidth]{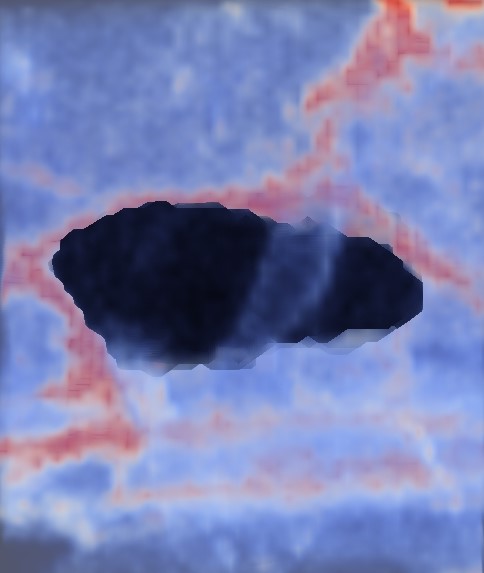}
	\end{minipage}
	\vspace{0.045cm}
	\begin{minipage}{0.2\textwidth}
		\centering
		\includegraphics[width=\textwidth]{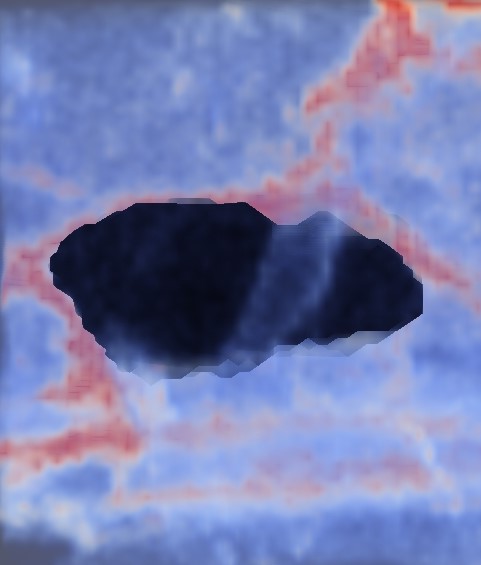}
	\end{minipage}
	\vspace{0.030cm}
	
	\begin{minipage}{0.2\textwidth}
		\centering
		\includegraphics[width=\textwidth]{membrane2tobesegmented.jpg}
	\end{minipage}
	\vspace{0.045cm}
	\begin{minipage}{0.2\textwidth}
		\centering
		\includegraphics[width=\textwidth]{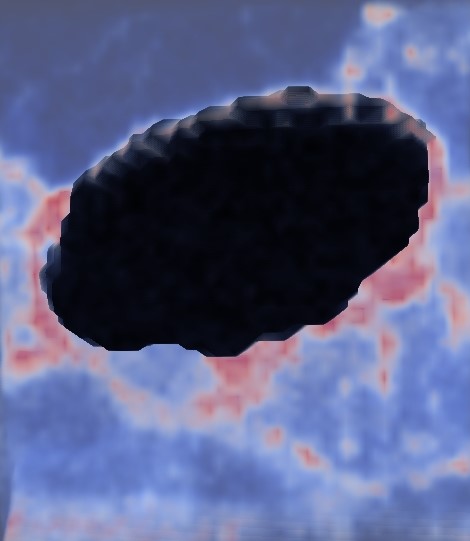}
	\end{minipage}
	\vspace{0.045cm}
	\begin{minipage}{0.2\textwidth}
		\centering
		\includegraphics[width=\textwidth]{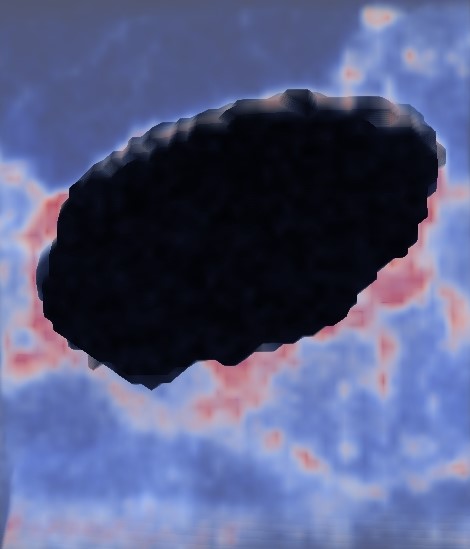}
	\end{minipage}
	\vspace{0.045cm}
	\begin{minipage}{0.2\textwidth}
		\centering
		\includegraphics[width=\textwidth]{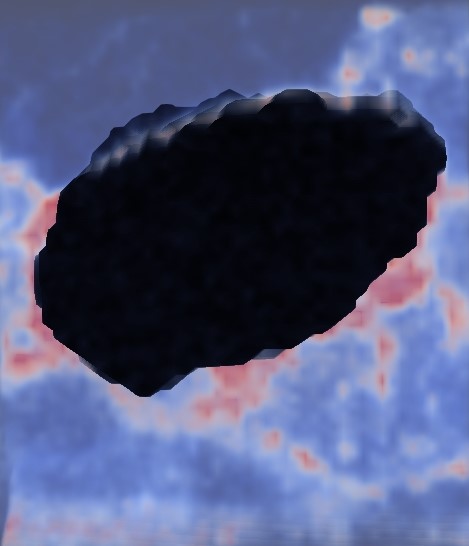}
	\end{minipage}
	\vspace{0.045cm}
	
	\begin{minipage}{0.2\textwidth}
		\centering
		\includegraphics[width=\textwidth]{membrane4tobesegmented.jpg}
	\end{minipage}
	\vspace{0.045cm}
	\begin{minipage}{0.2\textwidth}
		\centering
		\includegraphics[width=\textwidth]{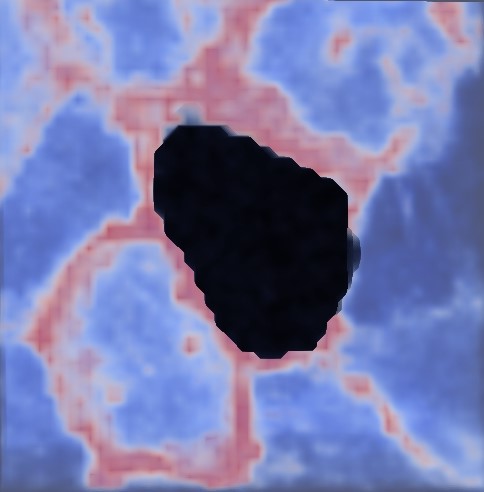}
	\end{minipage}
	\vspace{0.045cm}
	\begin{minipage}{0.2\textwidth}
		\centering
		\includegraphics[width=\textwidth]{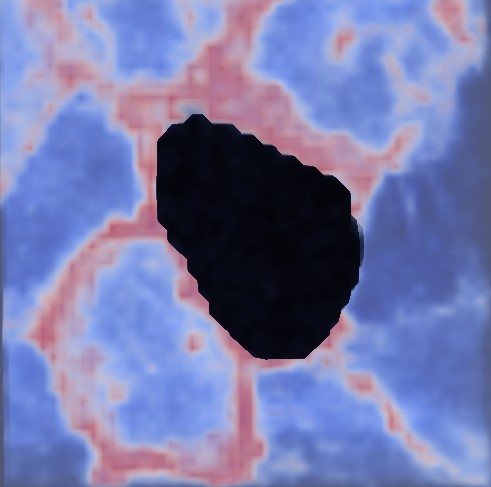}
	\end{minipage}
	\vspace{0.045cm}
	\begin{minipage}{0.2\textwidth}
		\centering
		\includegraphics[width=\textwidth]{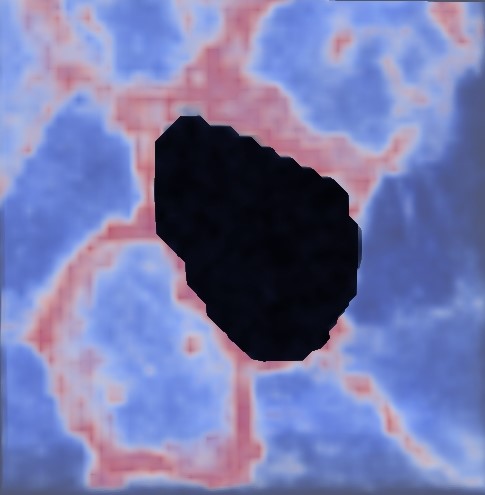}
	\end{minipage}
	\vspace{0.045cm}
	
	\caption{First column of this figure shows the $3$D microscopy images  
	of cell membrane to be segmented. The second, third, and fourth columns show 
	the result after the application of (\ref{first}) in 3D with $\delta = 0.8$  and $\vartheta = 0.2$,  $\delta = 0.7$  and $\vartheta = 0.3$, and $\delta = 0.6$  and $\vartheta = 0.4$, respectively.}
	\label{figure5}
\end{figure}
\noindent In \cite{4dpaper}, the first numerical scheme for $4$D image segmentation based on a generalized SUBSURF model was introduced. In this paper, we introduce and study a $4$D numerical scheme for the solution of the new $4$D segmentation model (\ref{first}) presented in the next section. This numerical method is based on the finite volume approach. The stability of the numerical scheme is also presented. Additionally, we introduced the local rescaling of the values of the segmentation function at each segmentation step to the interval $[0, 1]$. Furthermore, the application of the studied model to $3$D and $4$D image segmentation is studied. OpenMP and MPI parallelizations are employed for efficient computer implementation. For example, in the real application presented in the fifth and sixth experiments of Section \ref{chapterrealapp}, to process the 3D+time microscopy images, with dimensions $567\times577\times147\times70$, 1012 GB of memory was needed. This clearly shows that it may not be possible to process these images on a serial machine without parallel implementation utilizing the MPI. 
\subsection{Mathematical model}
\noindent Let $I^0 : \mathbb{O} \subset \mathbb{R}^4 \longrightarrow \mathbb{R}$, be the intensity function of a 4D image where $\mathbb{O} = \Omega \times [0,\theta_F]$, $\Omega \subset \mathbb{R}^3$, and $\theta_F$ represents the final real video time. For each real video time $\theta \in [0, \theta_F]$, let $C_\theta = \{c^\theta_m\}^{N_\theta}_{m = 1}$ denote the set of ``centers,'' where $c^\theta_m$ and $N_\theta$ represent the $m^{th}$ center and the total number of centers, respectively, at time $\theta$. Furthermore, for each $c^\theta_m \in C_\theta$ and $\theta \in [0, \theta_F]$, let $\alpha^\theta_m = \min \limits_{y~\in~B^\theta_m(c^\theta_m,r)} I^0(y,\theta)$, $\beta^\theta_m = \max \limits_{y~\in~B^\theta_m(c^\theta_m,r)} I^0(y,\theta)$, where $~B^\theta_m(c^\theta_m,r)$ is the $m^{th}$ ball at time $\theta$
 with radius $r$ centered at $c^\theta_m$, a given point (``center'') inside the object to be segmented. Then the threshold value (which is used for local thresholding) may be 
 chosen as ${TH}^\theta_m = \lambda ~ \alpha^\theta_m + (1 - \lambda) ~ \beta^\theta_m$, $\lambda \in [0, 1]$ and the ball radius may be chosen with respect to the approximate
 size of the object to be segmented. So, the idea (of local thresholding) is to set all intensity values in the local neighborhood of center $c^\theta_m$ to $\beta^\theta_m$ if they are above ${TH}^\theta_m$ and 
 $\alpha^\theta_m$ otherwise. Hence, the $4$D image intensity of the thresholded image within a ball of radius $r$ is defined by
 
 \begin{equation*}\label{thr}
 I^{TH}(y,\theta) =
	\begin{cases}
    \beta^\theta_m, \; I^0(y,\theta)  \geq {TH}^\theta_m,\\
	\alpha^\theta_m, \; otherwise,
	\end{cases}
\end{equation*}
where $y~\in~B^\theta_m(c^\theta_m,r)$.

\noindent Our new method is based on solution of the following generalized SUBSURF equation
\begin{eqnarray}\label{first}
   \frac{\partial u}{\partial t} = |\nabla u|\nabla \cdot \Big(G^0 \frac{\nabla u}{|\nabla u|} \Big)\text{ in } (0,T] \times \mathbb{O},
 \end{eqnarray}
 where $u(t,x)$ is the unknown 
 segmentation function, $x \in \mathbb{O}$, $t \in [0,T]$ is a ``segmentation time,'' $G^0 = g(\delta|\nabla G_{\sigma}\ast I^0| + \vartheta|\nabla G_{\sigma}\ast I^{{TH}}|)$, $g$ is the function introduced by Perona and Malik in \cite{pm} (see, e.g., equation (\ref{eg1})), $\delta, \vartheta \in [0, 1]$ determine the influence of information obtained from thresholded and original image intensities, 
 $G_{\sigma}$ is the smoothing kernel, e.g., the Gaussian function, $\triangledown G_\sigma *I^0$ and $\triangledown G_\sigma *I^{TH}$ are presmoothing of $I^0$ and $I^{TH}$ by convolution with $G_{\sigma}$, respectively. It is well known, see e.g., \cite{koend, witkin}, that the convolution with Gaussian function $G_{\sigma}$ is equivalent to solving the linear heat equation (LHE) obtained by setting $G^0 \equiv 1$ in equation (\ref{first}), with $t = \frac{1}{2}\sigma^2$. Thus, the presmoothing of $I^0$ and $I^{TH}$ by convolution with Gaussian function is realized by solving the LHE. Equation (\ref{first}) is accompanied by
Dirichlet boundary conditions
\begin{eqnarray}\label{first11}
 u(t,x) = u^D \text{ on } [0,T] \times \partial \mathbb{O},
\end{eqnarray}
and with the initial condition
\begin{eqnarray}\label{first33}
 u(0,x) = u^0(x)\text{ in } \mathbb{O}.
\end{eqnarray}
Without loss of generality, $u^D = 0$ is assumed. The initial condition in $4$D is constructed using equation (\ref{initialcond}), where $x \in \Omega$.

\subsection{Numerical discretization}\label{4dnumerics}
\noindent In this section, the details of time and space discretization is presented.
\subsubsection{Time discretization}
\noindent For time discretization of (\ref{first}), semi-implicit approach which guarantees
unconditional stability is used. Suppose that the (\ref{first})–(\ref{first3}) is solved in time interval $I = [0, T]$ and $N$ equal number of time steps. 
If $\tau = \frac{T}{N}$
denotes the time step, then the time discretization of (\ref{first}) is given by
\begin{eqnarray}\label{3rd}
  \frac{1}{\sqrt{\varepsilon^2 + |\nabla u^{n-1}|^2}} \frac{u^n - u^{n-1}}{\tau} = \nabla \cdot \bigg(G^0 \frac{\nabla u^n}{\sqrt{\varepsilon^2 + |\nabla u^{n-1}|^2}} \bigg),\\ \nonumber
 \end{eqnarray}
 where $\varepsilon$ is the regularization factor (Evans - Spruck \cite{evans}), 
 $u^0$ is given initial segmentation function, and $u^n$, $n = 1,\cdots, N$ is the solution of the model in segmentation time step $n$.
 \subsubsection{Space discretization}
 \noindent For space discretization, we start with introduction of some notations which will be used subsequently. We have
 adopted similar notations as those used in \cite{kmikula}. Let $\mathcal{T}_h$ denote
 finite volume mesh containing the doxels of $4$D image, while $V_{ijkl}$, $i = 1, \cdots, N_1$, $j = 1, \cdots, N_2$, $k = 1, \cdots, N_3$, $l = 1, \cdots, N_4$ denote
 each finite volume. For each $V_{ijkl} \in \mathcal{T}_h$, let $h_1$, $h_2$, $h_3$, $h_4$ be the size of the volumes in $x_1$, $x_2$, $x_3$, $x_4$ direction.
 Let the volume of $V_{ijkl}$ and its barycenter be denoted by $m(V_{ijkl})$ and $c_{ijkl}$, respectively. Let the approximate value of $u^n$ in 
 $c_{ijkl}$ be denoted by $u_{ijkl}^n$. For every $V_{ijkl} \in \mathcal{T}_h$, we have the following definitions:
 \begin{itemize}
     \item $N_{ijkl} = \{(p, q, r, s): p, q, r, s \in \{-1, 0, 1\},~ |p| + |q| + |r| + |s| = 1\}$ 
     \item $P_{ijkl} = \{(p, q, r, s): p, q, r, s \in \{-1, 0, 1\},~ |p| + |q| + |r| + |s| = 2\}$
 \end{itemize}
 For each $(p, q, r, s) \in N_{ijkl}$, denote the line connecting the center of $V_{ijkl}$ and the center of its neighbor 
 $V_{i+p,j+q,k+r,l+s}$ by $\sigma_{ijkl}^{pqrs}$ and  its length $m(\sigma_{ijkl}^{pqrs})$. We denote 
 the sides, its measure, and normal in coordinate directions of the finite volume $V_{ijkl}$ by $e_{ijkl}^{pqrs}$, $m(e_{ijkl}^{pqrs})$, and $\nu_{ijkl}^{pqrs}$, respectively. Furthermore, for each $(p, q, r, s) \in P_{ijkl}$, $y_{ijkl}^{pqrs}$ is defined by
 \begin{eqnarray*}
  y_{ijkl}^{pq00} &=& \frac{1}{4} \Big(c_{ijkl} + c_{i+p,j,k,l} + c_{i,j+q,k,l} + c_{i+p,j+q,k,l}\Big), \\
  y_{ijkl}^{0qr0} &=& \frac{1}{4} \Big(c_{ijkl} + c_{i,j+q,k,l} + c_{i,j,k+r,l} + c_{i,j+q,k+r,l}\Big), \\
  y_{ijkl}^{00rs} &=& \frac{1}{4} \Big(c_{ijkl} + c_{i,j,k+r,l} + c_{i,j,k,l+s} + c_{i,j,k+r,l+s}\Big), \\
  y_{ijkl}^{p0r0} &=& \frac{1}{4} \Big(c_{ijkl} + c_{i+p,j,k,l} + c_{i,j,k+r,l} + c_{i+p,j,k+r,l}\Big), \\
   y_{ijkl}^{p00s} &=& \frac{1}{4} \Big(c_{ijkl} + c_{i+p,j,k,l} + c_{i,j,k,l+s} + c_{i+p,j,k,l+s}\Big),\\
   y_{ijkl}^{0q0s} &=& \frac{1}{4} \Big(c_{ijkl} + c_{i,j+q,k,l} + c_{i,j,k,l+s} + c_{i,j+q,k,l+s}\Big).
\end{eqnarray*}
 The approximate value of $u^{n-1}$ in $y_{ijkl}^{pqrs}$,
 with $(p, q, r, s) \in P_{ijkl}$, is denoted by $u_{ijkl}^{pqrs}$; the time index is omitted because only the values from the time level $n - 1$ will be needed at these points.\\\\
 \noindent With these notations, integration of (\ref{3rd}) over finite volume $V_{ijkl}$ yields
 \begin{eqnarray}\label{4pt5}
  \int\limits_{V_{ijkl}} \frac{1}{\sqrt{\varepsilon^2 + |\nabla u^{n-1}|^2}} \frac{u^n - u^{n-1}}{\tau} dx= 
  \int\limits_{V_{ijkl}} \nabla \cdot \bigg(G^0 \frac{\nabla u^n}{\sqrt{\varepsilon^2 + |\nabla u^{n-1}|^2}} \bigg)dx.
 \end{eqnarray}
 Let the average value of $A_\varepsilon = \sqrt{\varepsilon^2 + |\nabla u^{n-1}|^2}$ in finite volume $V_{ijkl}$ be denoted by $\bar{A}_{\varepsilon,ijkl}^{n-1}$. If we consider the fact that
 $u^{n}$ and $u^{n-1}$ are asummed to be piecewise constant over the finite volume mesh on the left hand side of (\ref{4pt5}) and using the divergence theorem on the right hand side of (\ref{4pt5}), we obtain
 \begin{eqnarray}\label{4pt6}
  \frac{m(V_{ijkl})}{\bar{A}_{\varepsilon,ijkl}^{n-1}} \frac{u_{ijkl}^n - u_{ijkl}^{n-1}}{\tau} = 
   \sum\limits_{N_{ijkl}}~ 
  \int\limits_{e_{ijkl}^{pqrs}} G^0 \frac{\nabla u^n}{\sqrt{\varepsilon^2 + |\nabla u^{n-1}|^2}}\cdot \nu_{ijkl}^{pqrs} dS,
 \end{eqnarray}
 where beneath the summation sign, we used just $N_{ijkl}$ instead of $(p,q,r,s) \in N_{ijkl}$ to simplify notation. 
 If we approximate normal derivative $\nabla u^n \cdot \nu_{ijkl}^{pqrs}$ by $(u_{i+p,j+q,k+r,l+s}^n - u_{ijkl}^n)/ m(\sigma_{ijkl}^{pqrs})$
 and define $A_{\varepsilon,ijkl}^{pqrs;n-1}$ and $G_{ijkl}^{pqrs}$ to be the average of $A_\varepsilon$ and $G^0$ on $e_{ijkl}^{pqrs}$ then (\ref{4pt6}) reduces to
 \begin{eqnarray}\label{64th}
  m(V_{ijkl}) \frac{u_{ijkl}^n - u_{ijkl}^{n-1}}{\tau} = 
  \bar{A}_{\varepsilon,ijkl}^{n-1}\sum\limits_{N_{ijkl}}~ 
  m(e_{ijkl}^{pqrs}) G_{ijkl}^{pqrs} \frac{u_{i+p,j+q,k+r,l+s}^n - u_{ijkl}^n}{A_{\varepsilon,ijkl}^{pqrs;n-1}m(\sigma_{ijkl}^{pqrs})}.
 \end{eqnarray}
 Equation (\ref{64th}) can be rewritten as 
 \begin{eqnarray}\label{74th}
  u_{ijkl}^n =  u_{ijkl}^{n-1} +
  \frac{\tau}{m(V_{ijkl})} \bar{A}_{\varepsilon,ijkl}^{n-1}\sum\limits_{N_{ijkl}}~ 
  m(e_{ijkl}^{pqrs}) G_{ijkl}^{pqrs} \frac{u_{i+p,j+q,k+r,l+s}^n - u_{ijkl}^n}{A_{\varepsilon,ijkl}^{pqrs;n-1}m(\sigma_{ijkl}^{pqrs})},
 \end{eqnarray}
 which can be written in the form of system of equations
 \begin{eqnarray}\label{84th}
  \Bigg(1 + \frac{\tau}{m(V_{ijkl})} \bar{A}_{\varepsilon,ijkl}^{n-1}\sum\limits_{N_{ijkl}}~ 
  G_{ijkl}^{pqrs} \frac{m(e_{ijkl}^{pqrs}) }{A_{\varepsilon,ijkl}^{pqrs;n-1}m(\sigma_{ijkl}^{pqrs})}\Bigg) u_{ijkl}^n &-&\\ \nonumber
   \frac{\tau}{m(V_{ijkl})} \bar{A}_{\varepsilon,ijkl}^{n-1}\sum\limits_{N_{ijkl}}~ 
 G_{ijkl}^{pqrs} \frac{m(e_{ijkl}^{pqrs}) }{A_{\varepsilon,ijkl}^{pqrs;n-1}m(\sigma_{ijkl}^{pqrs})}u_{i+p,j+q,k+r,l+s}^n
  &=&u_{ijkl}^{n-1},\\ \nonumber
 \end{eqnarray}
 $i = 1, \cdots, N_1$, $j = 1, \cdots, N_2$, $k = 1, \cdots, N_3$, and $l = 1, \cdots, N_4$.
\noindent The  
   average values $G_{ijkl}^{pqrs}$, $A_{\varepsilon,ijkl}^{pqrs;n-1}$, and $\bar{A}_{\varepsilon,ijkl}^{n-1}$ either in doxels or on doxel sides
   are determined using the reduced diamond cell strategy (see \cite{kmikula}) adapted to $4$D.\\
   
\noindent In the sense of this reduced diamond cell approach, the approximate values of $u^{n-1}$
are obtained in the points $y_{ijkl}^{pqrs}$. These values are defined for each $(p, q, r, s) \in P_{ijkl}$ by
\begin{eqnarray*}
  u_{ijkl}^{pq00} &=& \frac{1}{4} \Big(u_{ijkl}^{n-1} + u_{i+p,j,k,l}^{n-1} + u_{i,j+q,k,l}^{n-1} + u_{i+p,j+q,k,l}^{n-1}\Big), \\
  u_{ijkl}^{0qr0} &=& \frac{1}{4} \Big(u_{ijkl}^{n-1} + u_{i,j+q,k,l}^{n-1} + u_{i,j,k+r,l}^{n-1} + u_{i,j+q,k+r,l}^{n-1}\Big), \\
  u_{ijkl}^{00rs} &=& \frac{1}{4} \Big(u_{ijkl}^{n-1} + u_{i,j,k+r,l}^{n-1} + u_{i,j,k,l+s}^{n-1} + u_{i,j,k+r,l+s}^{n-1}\Big), \\
  u_{ijkl}^{p0r0} &=& \frac{1}{4} \Big(u_{ijkl}^{n-1} + u_{i+p,j,k,l}^{n-1} + u_{i,j,k+r,l}^{n-1} + u_{i+p,j,k+r,l}^{n-1}\Big), \\
   u_{ijkl}^{p00s} &=& \frac{1}{4} \Big(u_{ijkl}^{n-1} + u_{i+p,j,k,l}^{n-1} + u_{i,j,k,l+s}^{n-1} + u_{i+p,j,k,l+s}^{n-1}\Big),\\
   u_{ijkl}^{0q0s} &=& \frac{1}{4} \Big(u_{ijkl}^{n-1} + u_{i,j+q,k,l}^{n-1} + u_{i,j,k,l+s}^{n-1} + u_{i,j+q,k,l+s}^{n-1}\Big).
\end{eqnarray*}
The components of the averaged gradient on $e_{ijkl}^{pqrs}$, $(p, q, r, s) \in N_{ijkl}$, are approximated by $2$D diamond
cell approach which use the values $u_{ijkl}^{pqrs}$ given above (see also \cite{kmikula}). Additionally, approximation of
the gradient on the face $e_{ijkl}^{pqrs}$ is denoted by $\nabla^{pqrs} u_{ijkl}^{n-1}$. This implies that

 \begin{eqnarray*}\label{9th}
  \nabla^{p000} u_{ijkl}^{n - 1} = \frac{1}{m(e_{ijkl}^{p000})}\int\limits_{e_{ijkl}^{p000}} \nabla u^{n-1} dx 
  ~~\approx \Bigg(p(u_{i+p,j,k,l}^{n-1} - u_{ijkl}^{n-1}) / h_1, (u_{ijkl}^{p,1,0,0} - u_{ijkl}^{p,-1,0,0}) / h_2,&&\\(u_{ijkl}^{p,0,1,0} - u_{ijkl}^{p,0,-1,0}) / h_3,(u_{ijkl}^{p,0,0,1} - u_{ijkl}^{p,0,0,-1}) / h_4\Bigg), 
 \end{eqnarray*}

 \begin{eqnarray*}\label{11th}
  \nabla^{0q00} u_{ijkl}^{n - 1} = \frac{1}{m(e_{ijkl}^{0q00})}\int\limits_{e_{ijkl}^{0q00}} \nabla u^{n-1} dx ~~\approx
  \Bigg((u_{ijkl}^{1,q,0,0} - u_{ijkl}^{-1,q,0,0}) / h_1, q(u_{i,j+q,k,l}^{n-1} - u_{ijkl}^{n-1}) / h_2,&&\\(u_{ijkl}^{0,q,1,0} - u_{ijkl}^{0,q,-1,0}) / h_3,(u_{ijkl}^{0,q,0,1} - u_{ijkl}^{0,q,0,-1}) / h_4\Bigg),
 \end{eqnarray*}

 \begin{eqnarray*}\label{12th}
  \nabla^{00r0} u_{ijkl}^{n - 1} = \frac{1}{m(e_{ijkl}^{00r0})}\int\limits_{e_{ijkl}^{00r0}} \nabla u^{n-1} dx ~~\approx
  \Bigg((u_{ijkl}^{1,0,r,0} - u_{ijkl}^{-1,0,r,0}) / h_1, (u_{ijkl}^{0,1,r,0} - u_{ijkl}^{0,-1,r,0}) / h_2,&&\\r(u_{i,j,k+r,l}^{n-1} - u_{ijkl}^{n-1}) / h_3,(u_{ijkl}^{0,0,r,1} - u_{ijkl}^{0,0,r,-1}) / h_4\Bigg),
 \end{eqnarray*}
\begin{eqnarray*}\label{13th}
  \nabla^{000s} u_{ijkl}^{n - 1} = \frac{1}{m(e_{ijkl}^{000s})}\int\limits_{e_{ijkl}^{000s}} \nabla u^{n-1} dx ~~\approx
  \Bigg((u_{ijkl}^{1,0,0,s} - u_{ijkl}^{-1,0,0,s}) / h_1, (u_{ijkl}^{0,1,0,s} - u_{ijkl}^{0,-1,0,s}) / h_2,&&\\(u_{ijkl}^{0,0,1,s} - u_{ijkl}^{0,0,-1,s}) / h_3,s(u_{i,j,k,l+s}^{n-1} - u_{ijk,l}^{n-1}) / h_4\Bigg).
 \end{eqnarray*}
If the same approach for computation of gradients of image intensities is used, then the following approximation of $G_{ijkl}^{pqrs}$ in (\ref{84th}) is obtained
 \begin{eqnarray}\label{10th}
  G_{ijkl}^{pqrs} = g\Big(\delta|\nabla^{pqrs} I_{\sigma;ijkl}^0| +
\vartheta|\nabla^{pqrs} I_{\sigma;ijkl}^{TH}|\Big),\\ \nonumber
\end{eqnarray}
 where $I_{\sigma}^0 = G_\sigma*I^0$, $I_{\sigma;ijkl}^0$ is the value of $I_\sigma^0$ in doxel $V_{ijkl}$, and $I_{\sigma}^{TH} = G_\sigma*I^{TH}$, $I_{\sigma;ijkl}^{TH}$ is the value of $I_\sigma^{TH}$ in doxel $V_{ijkl}$. Finally, incorporating $\varepsilon-$regularization, we obtain the following remaining terms in (\ref{84th})
 \begin{eqnarray}\label{111th}
A_{\varepsilon,ijkl}^{pqrs;n-1} = \sqrt{\varepsilon^2 + |\nabla^{pqrs} u_{ijkl}^{n-1}|^2},~ 
\bar{A}_{\varepsilon,ijkl}^{n-1} = \sqrt{\varepsilon^2 + \frac{1}{8}\sum\limits_{N_{ijkl}}|\nabla^{pqrs} u_{ijkl}^{n-1}|^2},
 \end{eqnarray}
 \noindent Equations (\ref{84th}) accompanied by the zero Dirichlet boundary condition, represent a linear system of equations which can 
  be solved efficiently, e.g., using the Successive Overrelaxation (SOR) method. 
 \begin{remark}[{\bf Local rescaling}]
  For each segmentation time step $n$, and $l = 1, \cdots, N_4$, let $C_l = \{c^l_m\}^{N_l}_{m = 1}$ denote the set of centers, where $c^l_m$ and $N_l$ represent the $m^{th}$ center and the total number of centers, respectively for each $l$. Let $~B^l_m(c^l_m,r)$ be a ball
 with radius $r$ and center $c^l_m$ ($c^l_m$ is a given point inside the object to be segmented), $\mu^{l}_m = \min \limits_{B^l_m(c^l_m,r)} u_{ijkl}^n$ and $\xi^{l}_m = \max \limits_{B^l_m(c^l_m,r)} u_{ijkl}^n$. 
  Then the locally rescaled version of $u_{ijkl}^n$ given by {\rm(\ref{84th})} within the ball $~B^l_m(c^l_m,r)$ is obtained by the 
  following relation
  \begin{eqnarray}\label{841}
   u_{ijkl}^{resc;n} = \frac{1}{\xi^{l}_m - \mu^{l}_m}(u_{ijkl}^{n} - \mu^{l}_m).
  \end{eqnarray} 
 \end{remark}
\noindent Consequently, we have that for each
  time step $n$, rescaled version $u_{ijkl}^{resc;n} \in [0,1]$ and it is used instead of $u_{ijkl}^{n-1}$ in (\ref{84th}) in the next segmentation time step. So in each segmentation step we solve the following system of equations representing the final formulation of our method:
 \begin{eqnarray}\label{new1}
  \Bigg(1 + \frac{\tau}{m(V_{ijkl})} \bar{A}_{\varepsilon,ijkl}^{n-1}\sum\limits_{N_{ijkl}}~ 
  G_{ijkl}^{pqrs} \frac{m(e_{ijkl}^{pqrs}) }{A_{\varepsilon,ijkl}^{pqrs;n-1}m(\sigma_{ijkl}^{pqrs})}\Bigg) u_{ijkl}^n &-&\\ \nonumber
   \frac{\tau}{m(V_{ijkl})} \bar{A}_{\varepsilon,ijkl}^{n-1}\sum\limits_{N_{ijkl}}~ 
 G_{ijkl}^{pqrs} \frac{m(e_{ijkl}^{pqrs}) }{A_{\varepsilon,ijkl}^{pqrs;n-1}m(\sigma_{ijkl}^{pqrs})}u_{i+p,j+q,k+r,l+s}^n
  &=&u_{ijkl}^{resc;n-1},\\ \nonumber
 \end{eqnarray}
 $i = 1, \cdots, N_1$, $j = 1, \cdots, N_2$, $k = 1, \cdots, N_3$, and $l = 1, \cdots, N_4$. The coefficients of (\ref{new1}) are computed using $u_{ijkl}^{resc;n-1}$ instead of $u_{ijkl}^{n-1}$.\\ 
 
 \noindent Finally, we note that if $h = h_1 = h_2 = h_3 = h_4$, then  $m(V_{ijkl}) = h^4$, $m(e_{ijkl}^{pqrs}) = h^3$, and $m(\sigma_{ijkl}^{pqrs}) = h$, and the equation (\ref{new1}) simplifies to
 \begin{eqnarray}\label{new}
  \Bigg(1 + \frac{\tau}{h^2} \bar{A}_{\varepsilon,ijkl}^{n-1}\sum\limits_{N_{ijkl}}~ 
  \frac{G_{ijkl}^{pqrs} }{A_{\varepsilon,ijkl}^{pqrs;n-1}}\Bigg) u_{ijkl}^n &-&\\ \nonumber
   \frac{\tau}{h^2} \bar{A}_{\varepsilon,ijkl}^{n-1}\sum\limits_{N_{ijkl}}~ 
  \frac{G_{ijkl}^{pqrs}}{A_{\varepsilon,ijkl}^{pqrs;n-1}}u_{i+p,j+q,k+r,l+s}^n
  &=&u_{ijkl}^{resc;n-1}.\\ \nonumber
 \end{eqnarray}
 Since in our implementation we used common $h$, in the sequel, we deal with the properties of the system (\ref{new}). All derived properties are simply adopted also to the system (\ref{new1}).\\
 
 \noindent The solution of the linear system given by equations (\ref{new}) is obtained by the successive overrelaxation (SOR) method as follows:
 \begin{eqnarray}\label{5.33}
 u_{ijkl}^{n(k+1)}= \omega\bar{u}_{ijkl}^{n(k+1)}+(1-\omega)u_{ijkl}^{n(k)}, 
\end{eqnarray}
where 
\begin{eqnarray}\label{gsee}
 \bar{u}_{ijkl}^{n(k+1)}=\nonumber \Bigg(u_{ijkl}^{resc;n-1}+\frac{\tau}{h^2} \bar{A}_{\varepsilon,ijkl}^{n-1}\sum\limits_{N_{ijkl}^1}~ 
  \frac{G_{ijkl}^{pqrs}}{A_{\varepsilon,ijkl}^{pqrs;n-1}}u_{i+p,j+q,k+r,l+s}^{n(k+1)}\\ \nonumber+ \frac{\tau}{h^2} \bar{A}_{\varepsilon,ijkl}^{n-1}\sum\limits_{N_{ijkl}^2}~ 
  \frac{G_{ijkl}^{pqrs}}{A_{\varepsilon,ijkl}^{pqrs;n-1}}u_{i+p,j+q,k+r,l+s}^{n(k)}\Bigg)
 &/&\\ \nonumber \Bigg(1+\frac{\tau}{h^2} \bar{A}_{\varepsilon,ijkl}^{n-1}\sum\limits_{N_{ijkl}}~ 
  \frac{G_{ijkl}^{pqrs}}{A_{\varepsilon,ijkl}^{pqrs;n-1}}\Bigg) 
\end{eqnarray}
denotes a Gauss-Seidel iterate and $\omega$ is the relaxation parameter; $N_{ijkl}^1$ is the set of neighbors whose new values are already known and $N_{ijkl}^2$ is the set of neighbors
whose new values are not yet known.\\

\noindent Furthermore, the RED-BLACK SOR method \cite{an, msarti, mittal} can be used to overcome the inherent serial difficulty of the classic SOR. The RED-BLACK SOR divides the domain into an alternating RED and BLACK elements so that for any $i, j, k, l$, if  $(i + j + k + l) ~\%~ 2 == 0$, the element is RED and if  $(i + j + k + l) ~\%~ 2 == 1$, the element is BLACK. Updating the BLACK element’s solution requires the knowledge of adjacent RED elements and vice versa. Consequently, the entire computation is divided into two phases: RED and BLACK elements’ update. All RED elements are updated simultaneously in the first phase, and the same algorithm can be applied to the calculation of BLACK elements in the second phase. Thus, the solution of the linear system given by equations (\ref{new}) can be obtained by the RED-BLACK SOR method as follows: 
\begin{itemize}
    \item First phase: If $(i + j + k + l) ~\%~ 2 == 0$ i.e., if i + j + k + l is even, then update all RED elements in equation (\ref{5.33}) simultaneously using
BLACK elements’ values from the previous iteration;
    \item Second phase: If $(i + j + k + l) ~\%~ 2 == 1$ i.e., if i + j + k + l is odd, then update all BLACK elements in equation (\ref{5.33}) simultaneously using
RED elements' values from first phase.
\end{itemize}
\subsection{Stability of the numerical scheme}
\noindent In this section, we present a short proof that the linear system given by equation {\rm(\ref{new})} has a unique solution and that numerical scheme employed is unconditionally stable. The method or technique of proof presented in \cite{hms, app2} has been adopted. 
\begin{df}
The semi-implicit finite volume scheme given by equation {\rm(\ref{new})} for solving equation {\rm(\ref{first})} is unconditionally stable if for each $\epsilon > 0$, $\tau > 0$ and $~n \in \{1,\cdots, N\}$, the following inequality (the discrete minimum–maximum principle) holds
\begin{eqnarray}\label{mmp}
 \min \limits_{V_{ijkl}~\in~\mathcal{T}_h} u_{ijkl}^{resc;n - 1}\leq \min \limits_{V_{ijkl}~\in~\mathcal{T}_h} u_{ijkl}^n 
 \leq \max \limits_{V_{ijkl}~\in~\mathcal{T}_h} u_{ijkl}^n \leq \max \limits_{V_{ijkl}~\in~\mathcal{T}_h} u_{ijkl}^{resc;n - 1}.
\end{eqnarray}
\end{df}
\begin{thm}
The linear scheme given by equation {\rm(\ref{new})} has a unique solution $u_{ijkl}^n$ and is unconditionally stable for each $\epsilon > 0$, $\tau > 0$ and $~n \in \{1,\cdots, N\}$.
\end{thm}
\begin{proof}
The equation (\ref{new}) together with Dirichlet boundary condition is a system of linear equations with square matrix whose off diagonal elements are given by 
\begin{eqnarray}
 -~\frac{\tau}{h^2} \bar{A}_{\varepsilon,ijkl}^{n-1}~ 
  \frac{G_{ijkl}^{pqrs}}{A_{\varepsilon,ijkl}^{pqrs;n-1}}~,~ (p, q, r, s) \in N_{ijkl}~.
\end{eqnarray}
Also, the diagonal elements given by 
\begin{eqnarray}
 1 + \frac{\tau}{h^2} \bar{A}_{\varepsilon,ijkl}^{n-1}\sum\limits_{N_{ijkl}}~ 
  \frac{G_{ijkl}^{pqrs}}{A_{\varepsilon,ijkl}^{pqrs;n-1}}
\end{eqnarray}
are non-negative and dominate the sum of absolute value of the nondiagonal elements in each row. Hence, the matrix of the linear system (\ref{new}) is a strictly diagonally dominant M-matrix. Consequently, the existence of a unique solution of (\ref{new}) is guaranteed \cite{sddm1, sddm2}. Next, we show that the scheme is unconditionally stable. For this, it is enough to show that equation (\ref{new}) satisfies (\ref{mmp}). Clearly, equation (\ref{new}) is same as 
\begin{eqnarray}\label{741th}
  u_{ijkl}^n +
  \frac{\tau}{h^2} \bar{A}_{\varepsilon,ijkl}^{n-1}\sum\limits_{N_{ijkl}}~ 
  G_{ijkl}^{pqrs} \frac{u_{ijkl}^n - u_{i+p,j+q,k+r,l+s}^n}{A_{\varepsilon,ijkl}^{pqrs;n-1}}=  u_{ijkl}^{resc;n-1}.
 \end{eqnarray}
 Let
\begin{eqnarray*}
 \max \limits_{V_{abcd}~\in~\mathcal{T}_h} u_{abcd}^n =u_{ijkl}^n.
\end{eqnarray*}
Then from (\ref{741th}) we have that
\begin{eqnarray*}
 \frac{\tau}{h^2} \bar{A}_{\varepsilon,ijkl}^{n-1}\sum\limits_{N_{ijkl}}~ 
  G_{ijkl}^{pqrs} \frac{u_{ijkl}^n - u_{i+p,j+q,k+r,l+s}^n}{A_{\varepsilon,ijkl}^{pqrs;n-1}}\geq0.
\end{eqnarray*}
Hence,
\begin{eqnarray}\label{5.25}
 \max \limits_{V_{abcd}~\in~\mathcal{T}_h} u_{abcd}^n \leq u_{ijkl}^{resc;n-1} \leq \max \limits_{V_{abcd}~\in~\mathcal{T}_h} u_{abcd}^{resc;n-1}.
\end{eqnarray}
Using similar arguments, we have that
\begin{eqnarray}\label{5.26}
 \min \limits_{V_{abcd}~\in~\mathcal{T}_h} u_{abcd}^{resc;n-1} \geq u_{ijkl}^{resc;n-1} \geq \min \limits_{V_{abcd}~\in~\mathcal{T}_h} u_{abcd}^{n}.
\end{eqnarray}
Equations (\ref{5.25}) and (\ref{5.26}) implies that
\begin{eqnarray}\label{5.27}
 \min \limits_{V_{ijkl}~\in~\mathcal{T}_h} u_{ijkl}^{resc;n-1} \leq \min \limits_{V_{ijkl}~\in~\mathcal{T}_h} u_{ijkl}^{n} 
 \leq \max \limits_{V_{ijkl}~\in~\mathcal{T}_h} u_{ijkl}^{n} \leq \max \limits_{V_{ijkl}~\in~\mathcal{T}_h} u_{ijkl}^{resc;n-1},
\end{eqnarray}
which yield equation (\ref{mmp}). Hence, the numerical scheme given by equation
(\ref{new}) is unconditionally stable.
\end{proof}
\subsection{Experimental order of convergence}\label{eocsection}
\noindent Assuming that the error of a scheme in some given norm is proportional to some power, $\alpha$, of the grid size $h$. Then we have that $Err(h) = Ch^\alpha$ where $C$ is a contant of proportionality.
If half of the grid size is considered, i.e., $h  := \frac{h}{2}$ then\\ $Err(\frac{h}{2}) = C(\frac{h}{2})^\alpha$. 
Consequently, 
\begin{eqnarray}
\alpha = \log_ 2 \Big(\frac{Err(h)}{Err(\frac{h}{2})}\Big).
\end{eqnarray}
The $\alpha$ is called the \emph{experimental order of convergence} $(EOC)$ and is determined
by comparing numerical solutions and exact solutions on subsequently refined grids \cite{cmssga}.

\noindent We now test our method using the exact solution (see also \cite{cmssga})
\begin{eqnarray}\label{exactsol}
u(x_1, x_2, x_3, x_4, t) = \frac{x_1^2 + x_2^2 + x_3^2 + x_4^2 - 1}{6} + t.
\end{eqnarray}
of the level set equation
\begin{eqnarray}\label{levelsetmcf}
u_t = |\nabla u|\nabla \cdot \frac{\nabla u}{|\nabla u|}
\end{eqnarray}
and consider Dirichlet boundary conditions given by this exact solution.\\
This problem is solved in the spatial domain $\Omega = [−1.25,~ 1.25]^4$ and in the time interval $T = 0.0625$. We
have taken subsequent grid refinement $N_1 = N_2 = N_3 = N_4 =
10,~ 20,~ 40,~ 80,~ 160$ and considered number of discrete time steps $1,~ 4,~ 16,~ 64,~ 256$ respectively. The time step $\tau$ is chosen proportionally
to $h^2$ and we measure errors in $L_2((0, T), L_2(\Omega))-$norm. Table \ref{gridrefinementerr} shows errors in $L_2((0, T), L_2(\Omega))-$norm for refined grids and $\epsilon = h^2$. 

\begin{table}[H]
\centering
\caption{Errors in $L_2((0, T), L_2(\Omega))-$norm, and EOC comparing numerical and exact solution (\ref{exactsol}).} 
\label{gridrefinementerr}
\begin{tabular}{|c |c |c |c |c| c| c|}
\hline
$n$ & $h$ & final step & Error ($\epsilon = h^2$) & EOC\\ [0.5ex] 
\hline 
10 & $0.25$ & 1 & $1.680403e-2$&\\ 
20 & $0.125$ & 4 &$4.653133e-3$ & 1.852533\\
40 & $0.0625$ & 16 &  $1.208771e-3$& 1.944661\\
80 & $0.03125$ & 64 & $3.029600e-4$ & 1.996342\\
160 & $0.015625$ & 256  & $8.340820e-5$ & 1.860866\\
 [1ex]
\hline 
\end{tabular}

\end{table}
\noindent We observe that $\alpha \approx 2$ as the grids get more refined, implying that the method converges with an order of almost two. Thus, we can conclude that the scheme is reliable and can be used in practical applications.
\subsection{Brief overview of computer implementation}
\noindent The following steps provide an overview of the implementation steps for the new $4$D model. 
\begin{itemize}
    \item Read input 3D+time image together with the corresponding centers of cells.
    \item Using the input centers, generate initial segmentation function (or initial condition).
    \item Locally rescale the initial segmentation function to interval $[0,1]$.
    \item Using the input centers, perform local thresholding of $4$D image.
    \begin{enumerate}
        \item Compute the  
   coefficients $G_{ijkl}^{pqrs}$, $A_{\varepsilon,ijkl}^{pqrs;n-1}$, and $\bar{A}_{\varepsilon,ijkl}^{n-1}$ either in doxels or on doxel sides using the locally rescaled segmentation function.
    \item Solve the linear system given by (\ref{new}).
    \item Locally rescale the computed segmentation function to interval $[0,1]$.
    \end{enumerate}
    \item Repeat steps $1$, $2$, and $3$ until the total number of segmentation steps is reached.
    \item Output the result of segmentation
\end{itemize}

\noindent For complete serial and parallel implementation in C programming language, see 
\url{https://github.com/88MARK08/4D-image-segmentation-algorithm}.

\subsection{Parallel implementation using OpenMP}
\noindent OpenMP is a multi-threading implementation. In C/C++, omp.h header file includes all OpenMP functions. In our implementation of the new $4$D model, \emph{\#pragma omp parallel private\{$\cdots$\}} is used to instruct the OpenMP system
to divide tasks among the working threads. In the for loops,
the first loop is the loop for the real-time length $\theta$. In other words, we split a series of $3$D volumes among working threads. Furthermore, \emph{\#pragma omp parallel for private\{$\cdots$\} reduction(operator:variable)} is used to accomplish reduction operations. For instance, \emph{reduction (operator: variable)} is used to specify that the operation given by the ``\emph{operator}'' should be performed on the values of the ``\emph{variable}'' from all threads at the end of the parallel construct. Finally, to measure CPU time in parallel implementation, \emph{omp\_get\_wtime()} function is used.
\vskip 0.15cm
\noindent In table \ref{speedupopenmp}, we present a comparison of CPU times with OpenMP parallel implementation and serial implementation. In this experiment, a PC with 8192 MB RAM and processor: Intel(R) Core(TM) i7-7700HQ CPU @ 2.80 GHz (8 CPUs), $\sim2.8$ GHz was used.

\begin{table}[H]
\centering
\caption{Computing times and speed-up of OpenMP parallel
program for computing the EOC in Section \ref{eocsection}, with $n = 40$.} 
\label{speedupopenmp}
\begin{tabular}{|c |c c c c|}
\hline
\# threads & 1 & 2 & 4 & 8 \\ [0.5ex] 
\hline 
time (secs) & $225.8700$ & $119.0000$ & $71.9270$ & $42.3050$ \\ 
speed-up & $0$ & $1.898$ & $3.140$ & $5.339$\\ 
[1ex] 
\hline 
\end{tabular}

\end{table}
  
\subsection{Parallel implementation using MPI}
\noindent The two main goals of parallel program implementation are handling huge amounts of data that cannot be placed in the memory of a single serial computer and the shortest possible program execution time.
Assuming that in terms of execution time, a fraction $P$ of a program can be parallelized. 
In an ideal case, while executing a parallel program on $n_p$ processors, the execution time will be $1 − P + \frac{P}{n_p}$. Furthermore, the theoretical speed-up according to the Amdahl's law \cite{amdahl} is given by
$\frac{1}{(1 − P) ~+~ \frac{P}{n_p}}$.
Consequently, if only 90\% of the program can be parallelized, for example, then with infinitely many processors
the maximal speed-up (estimated from Amdahl's law by $\frac{1}{1 - P}$) cannot exceed $10$. To minimize the time spent 
on communication, it is necessary to require that the data transmitted (e.g., multidimensional arrays) 
be contiguous in memory. This is to ensure that data are exchanged directly among processes in one message using 
only one call of MPI send and receive subroutines. In this section, discrete $4$D image is represented by a four-dimensional array indexed by $i, j, k, l$, with $i = 1$, $\cdots$, $N_1$, $j = 1$, $\cdots$, $N_2$, $k = 1$, $\cdots$, $N_3$, $l = 1$, $\cdots$, $N_4$. The discrete computational domain including the boundary conditions is given by  $i = 0$, $\cdots$, $N_1 + 1$, $j = 0$, $\cdots$, $N_2 + 1$, $k = 0$, $\cdots$, $N_3 + 1$, $l = 0$, $\cdots$, $N_4 + 1$.
The boundary positions, in the computational domain, with $i = 0$, $i = N_1+ 1$, $j = 0$, $j = N_2 + 1$, $k = 0$, $k = N_3 + 1$, $l = 0$, $l = N_4 + 1$ are reserved for Dirichlet boundary conditions and all the inner doxel positions correspond to the original 4D image. Let $n_p$ be the number of processes, then to distribute the 4D data, define $n_4 = \lceil \frac{N_4}{n_p}\rceil$, $n_4 ^{last} = N_4 −(n_p − 1)n_4$, $n_3 = N_3$, $n_2 = N_2$, $n_1 = N_1$. Hence, the processes with rank from $0$ to $n_p − 2$ deal with part of the discrete 4D image which is the series of 3D volumes given by the array with indices $i = 0$, $\cdots$, $n_1 + 1$, $j = 0$, $\cdots$, $n_2 + 1$,
$k = 0$, $\cdots$, $n_3 + 1$, indexed locally by $l$ in the range $l = 0$, $\cdots$, $n_4 + 1$. Additionally, on the last process with rank $n_p − 1$, the index $l$ of the last 3D volume is $n_4^{last} + 1$ instead of $n_4 + 1$. The merging of all 3D volumes for $l = 1, \cdots, n_4$ ($n_4^{last}$ on the last process)
from all processes gives the non-distributed complete 4D image. For the purpose of iterative solution of the linear system and computation of its coefficients, the data overlap is needed. The overlap with the necessity of 
information exchange between neighbouring processes is given by the slices $n_4$, $n_4 + 1$ and slices $0,~ 1$ of the subsequent processes. It is good to mention at this point that in our computer implementation, $iMax =  n_1 + 2$, $jMax =  n_2 + 2$, $kMax =  n_3 + 2$, $lMax =  n_4 + 2$, $N = proc\_lMax = n_4$. Furthermore, the following relationship is used (see Listing \ref{4d21d}) to transform $4$D to $1$D array:
\begin{eqnarray*}
 T_{1D}(i,j,k,l) = l * iMax * jMax * kMax + k * iMax * jMax + j * iMax + i.
\end{eqnarray*}
\newpage
\begin{lstlisting}[language=C++, label={4d21d}, caption={Transformation of 4D array to 1D array}] 
int ijkl(int i, int j, int k, int l)
{
	return l * iMax * jMax * kMax + k * iMax * jMax + j * iMax + i;
	// or equivalently return ((l * kMax + k) * jMax + j) * iMax + i
}

\end{lstlisting}

\noindent In each segmentation time step, the solution values are updated as the linear system is solved iteratively. In each
iteration, four of these neighbors should be known already. Consequently, each consecutive process must wait until its preceding process is completed to get the unknown neighbors' values updated in a parallel run. The RED-BLACK SOR method (see, e.g., \cite{an, msarti, mittal}) is employed to do away with this dependency.
In the RED-BLACK SOR, all doxels in the computational domain are split into RED elements, given
by the condition that the sum of its indices is an even number, and BLACK elements, given by the condition that the sum of its indices is an odd number. 
So, the eight neighbors  of RED elements are BLACK elements and the
value of RED elements depends only on those of the BLACK elements, and vice
versa \cite{msarti}. As a result of this method, one SOR iteration is split into two steps. In the first step,
RED elements are updated and BLACK elements are updated in the second step. This splitting is perfectly
parallelizable \cite{msarti}.\\

\noindent After computation of one RED-BLACK SOR iteration for the RED elements on
every parallel process, RED updated values have to be exchanged in the overlapping
regions, and then one iteration for BLACK elements can be computed. The data
exchange is implemented using non-blocking MPI\_Isend
and MPI\_Irecv subroutines. 
For the residual's computation, partial information
from all the processes are collected and send to all processes to check the
stopping criterion by every process.\\

\noindent We note that all parallel computations were performed on a Linux cluster comprising six computational servers (nodes) and 192 processors. Each computational node has $252$ GB of memory and 32 processors; thus, the cluster has $1512$ GB of memory available for computations. Additionally, we used processors belonging to one of the computational nodes for the experiment involving computing the EOC in Section \ref{eocsection},  whose results are shown in Table \ref{table:speedup22}. Furthermore, we used the six servers to ensure sufficient memory resources for the fifth and sixth experiments in Section \ref{chapterrealapp}, which involves solving a linear system with $3~ 366 ~466 ~110$ unknowns.\\

\noindent Table \ref{table:speedup22} shows a comparison of CPU times with MPI parallel implementation and serial implementation. In this experiment, processors in one of the six servers in the Linux cluster were utilized.
\vskip 0.2truecm
\begin{table}[H]
\centering
\caption{Computing times and speed-up of MPI parallel
program running on 1 to 16 processors for computing the EOC in Section \ref{eocsection}, with $n = 80$.} 
\label{table:speedup22}
\begin{tabular}{|c |c c c c c|}
\hline
\# processors & 1 & 2 & 4 & 8 & 16\\ [0.5ex] 
\hline 
time (secs) & $41681.08$ & $20976.89$ & $10675.87$ & $5487.309$ & $3077.166$\\ 
speed-up & $0$ & $1.987$ & $3.904$ & $7.596$ & 13.545\\ 
[1ex] 
\hline 
\end{tabular}

\end{table}

\noindent Table \ref{table:speedup22} shows a linear speed-up with $2$, $4$, and $8$ processors. However, with 16 processors, the speed-up dropped from approximately 16 (expected) to 13.545. This drop may be attributed to the amount of time spent on communication between processes. For instance, one, three, seven, and 15 2-way  communications are involved with two, four, eight, and 16 processors, respectively. Thus, we may conclude that the communication time impacts the speed-up of the MPI parallel program.\\
\section{Application to 4D image segmentation}\label{chapterrealapp}
\noindent In this section, the new $4$D method is tested on artificially generated $3$D+time videos and applied to real data representing 3D+time microscopy images of cell nuclei within the zebrafish pectoral fin and hind-brain.\\

\noindent In the first experiment (see \url{https://doi.org/10.5281/zenodo.5513089} for the videos), $3$D+time video of one sphere was artificially generated. The goal of this simple experiment is to show that the 4D method (\ref{first}) can approximate a missing shape in a 3D+time video. This is because SUBSURF models can complete a missing part of an object.  In the 4D case, the method is expected to complete a missing volume. Thus, from a $3$D+time video of 20 frames, we removed time frames 5, 10, and 15 and tried reconstructing the entire video using our 4D model. In Figure \ref{m_artificial_image_at_t1-20}, the first column shows the $3$D frame of the video at the time step $10$, the second and third columns show their corresponding reconstruction using the 4D method; the third column shows the result when the sphere at volume/frame 10 is removed. We note that in the second column, all frames were considered in the segmentation. The results of the segmentation shown in columns two and three are colored in blue. Furthermore, the second column of Figure \ref{m_artificial_image_at_t1-20} shows that the segmentation results of this moving sphere are good. However, the third column of this figure shows that the segmentation result of the sphere at the missing volume/frame is an ellipsoidal shape instead of a spherical shape. The approximation of a sphere with an ellipsoid is a good result that can be very useful during cell tracking. Thus, we can conclude from this experiment that our 4D method can approximate a missing shape in 3D+time.\\

\begin{figure}[h]
\centering
	\begin{minipage}{0.24\textwidth}
		\centering
		\includegraphics[width=\textwidth]{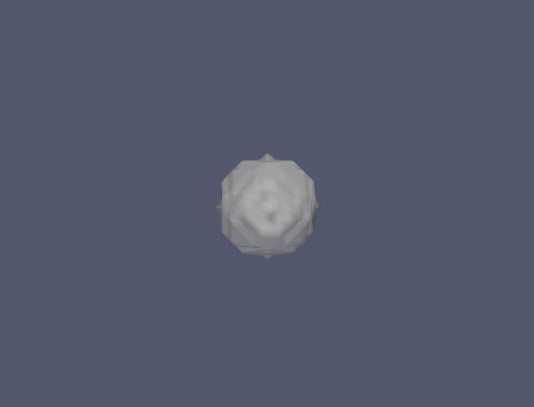}
	\end{minipage}
	\vspace{0.035cm}
	\begin{minipage}{0.24\textwidth}
		\centering
		\includegraphics[width=\textwidth]{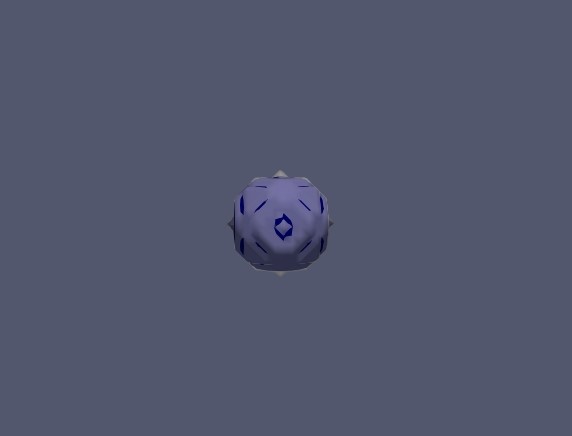}
	\end{minipage}
	\vspace{0.035cm}
	\begin{minipage}{0.24\textwidth}
		\centering
		\includegraphics[width=\textwidth]{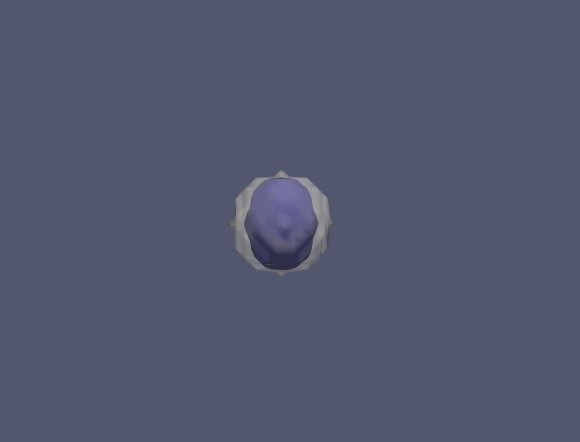}
	\end{minipage}
	\vspace{0.035cm}

	\caption{First column of this figure shows the $3$D frame of 4D image of a sphere at time step $10$, the second column shows the corresponding reconstruction using (\ref{first}), and the third column shows the result when the sphere at volume/frame 10 is removed.}
	\label{m_artificial_image_at_t1-20}
\end{figure}

\noindent In the second experiment (see \url{https://doi.org/10.5281/zenodo.5513089} for the videos), $3$D+time videos of spheres were artificially generated. There are four spheres in time steps $1-3$, seven spheres in time steps $4-17$, and five spheres in time steps $18-20$. In Figure \ref{artificial_image_at_t1-20}, the first row shows the $3$D frames of the video at the time steps $1$ (first column), $10$ (second column), and $20$ (third column), and the second row shows their corresponding reconstruction using the new method. The results of the segmentation are colored in blue. Furthermore, the segmentation results of the $3$D+time videos of these spheres show a very good performance of the 4D method (\ref{first}) on this artificial dataset.

\begin{figure}[h]
	\centering
	\begin{minipage}{0.24\textwidth}
		\centering
		\includegraphics[width=\textwidth]{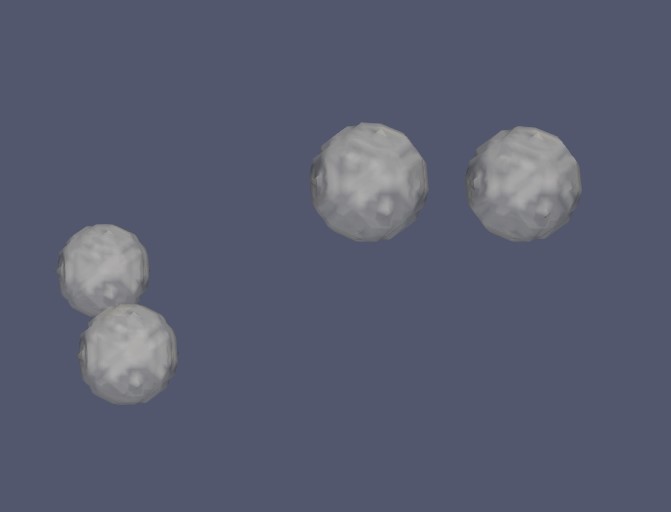}
	\end{minipage}
	\vspace{0.045cm}
	\begin{minipage}{0.24\textwidth}
		\centering
		\includegraphics[width=\textwidth]{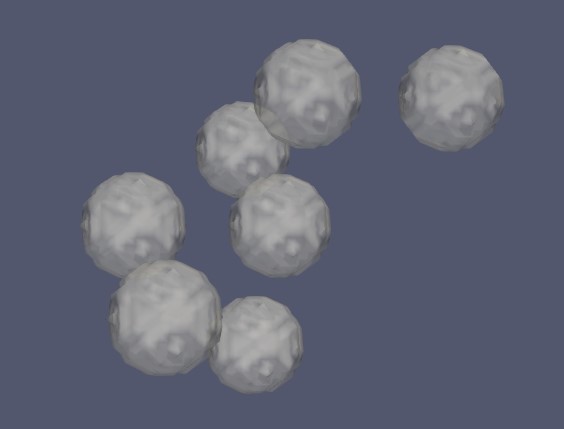}
	\end{minipage}
	\vspace{0.045cm}
	\begin{minipage}{0.24\textwidth}
		\centering
		\includegraphics[width=\textwidth]{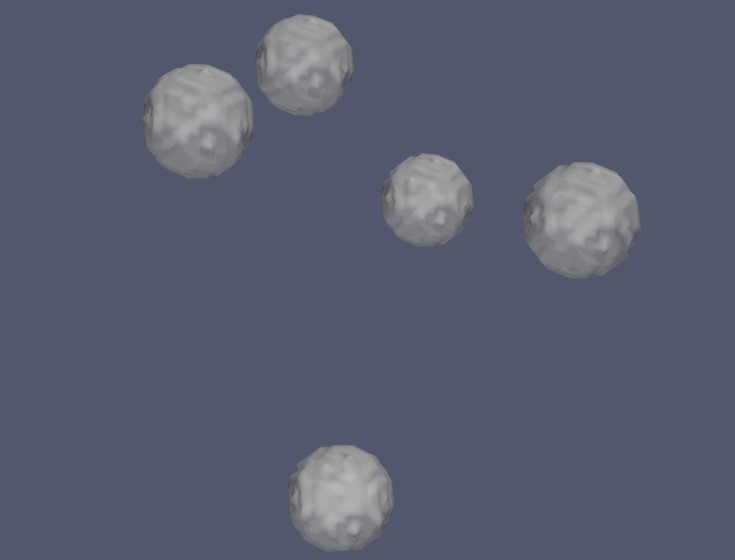}
	\end{minipage}
	\vspace{0.045cm}
	
	\begin{minipage}{0.24\textwidth}
		\centering
		\includegraphics[width=\textwidth]{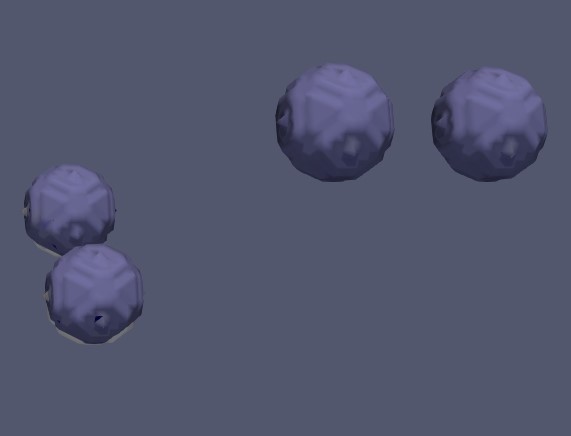}
	\end{minipage}
	\vspace{0.045cm}
	\begin{minipage}{0.24\textwidth}
		\centering
		\includegraphics[width=\textwidth]{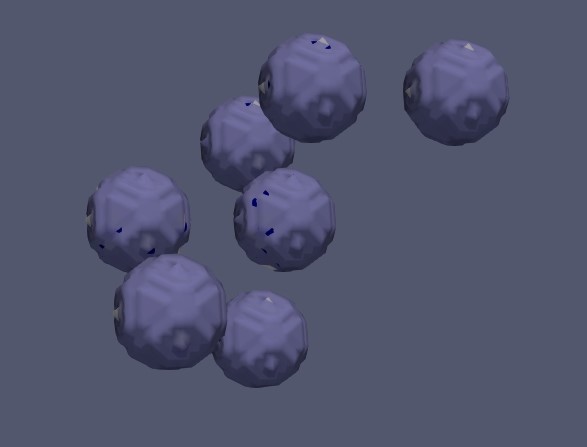}
	\end{minipage}
	\vspace{0.045cm}
	\begin{minipage}{0.24\textwidth}
		\centering
		\includegraphics[width=\textwidth]{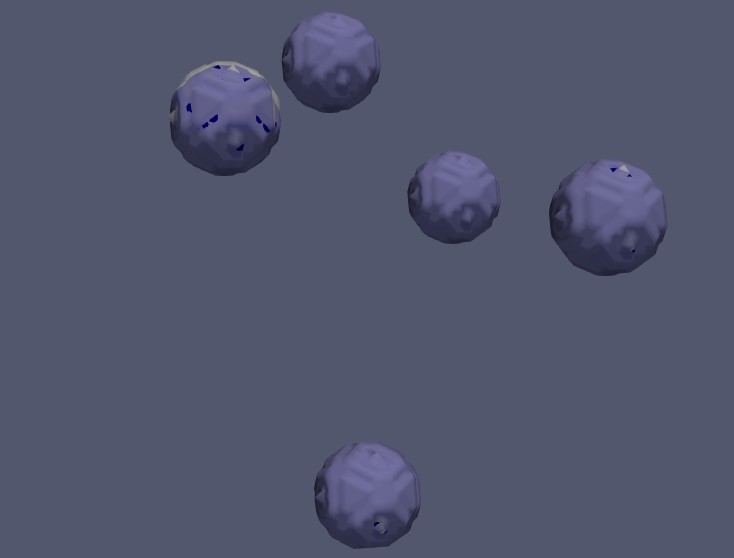}
	\end{minipage}
	\vspace{0.045cm}
	
	\caption{First row of this figure shows the $3$D frames of 4D image of spheres at time steps $1$ (first column), $10$ (second column), and $20$ (third column) and the second row shows their corresponding reconstruction using (\ref{first}).}
	\label{artificial_image_at_t1-20}
\end{figure}

\noindent In the following two experiments, we present the results of experiments with data from zebrafish hindbrain. 3D microscopy images of cell nuclei in the hindbrain of developing
zebrafish embryos were provided by Mageshi Kamaraj from the group of Nadine Peyri\'{e}ras (CNRS BioEmergences, France).\\

\noindent In the subsequent two experiments, seven arbitrary cell nuclei centers in zebrafish hindbrain were selected, and 23 time frames were considered. For each time frame, we considered a small $30 \times 30 \times 30$ computational domain around each of these seven nuclei centers. The motivation for this construction is to reduce the memory requirement of our serial implementation and ensure easy visualization. Furthermore, image intensities from the original dataset are copied to these small computational domains around nuclei centers in each time frame. Hence, when the 3D volumes containing these seven small cubes are put together, we obtain the 3D+time image shown in Figure \ref{7_4d_inputimage}.\\   

\noindent In the third experiment, using the 3D method case presented in \cite{ubaetal}, we performed 3D segmentation of these seven cell nuclei within the small $30 \times 30 \times 30$ computational domain in each 3D volume containing these seven small cubes. The 3D segmentation results of these cell nuclei in each 3D volume are put together over time and used as an input $3$D+time image to the 4D method (first column of Figure \ref{7_4d_inputimage}).
This is another artificial experiment that is more closer to experiments with real data. Hence, in generating the 3D+time image for this experiment, we consider shapes closer to real cell shapes than spheres. We note that in the first, second, and third experiments, the following pairs of parameters $(\delta = 1, \vartheta = 0)$, $(\delta = 0, \vartheta = 1)$, and $(\delta = 0.5, \vartheta = 0.5)$ in model (\ref{first}) yield the same result. This is because the thresholded and original image intensities are the same.\\

\noindent First image of Figure \ref{7_4d_inputimage} shows the visualization of seven zebrafish cell nuclei, segmented in $3$D and put together to form a $3$D+time image. Second image of Figure \ref{7_4d_inputimage} shows the corresponding segmentation result using the $4$D model (\ref{first}). The segmentation results are colored in blue and show accurate correspondence of the segmented moving cells in 4D with the original 4D image.

\begin{figure}[h]
\centering
	\begin{minipage}{0.44\textwidth}
		\centering
		\includegraphics[width=\textwidth]{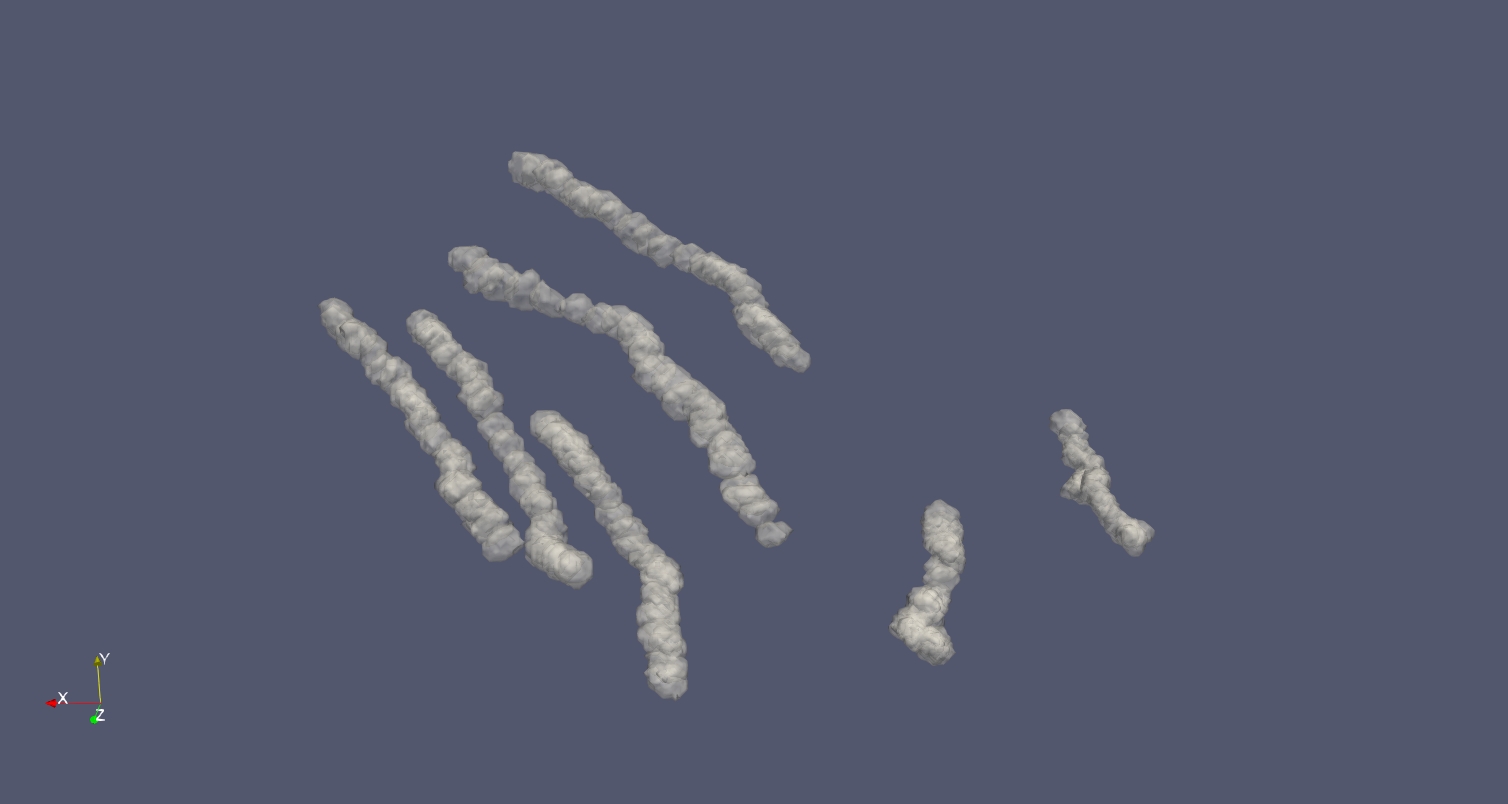}
	\end{minipage}
	\vspace{0.045cm}
	\begin{minipage}{0.44\textwidth}
		\centering
		\includegraphics[width=\textwidth]{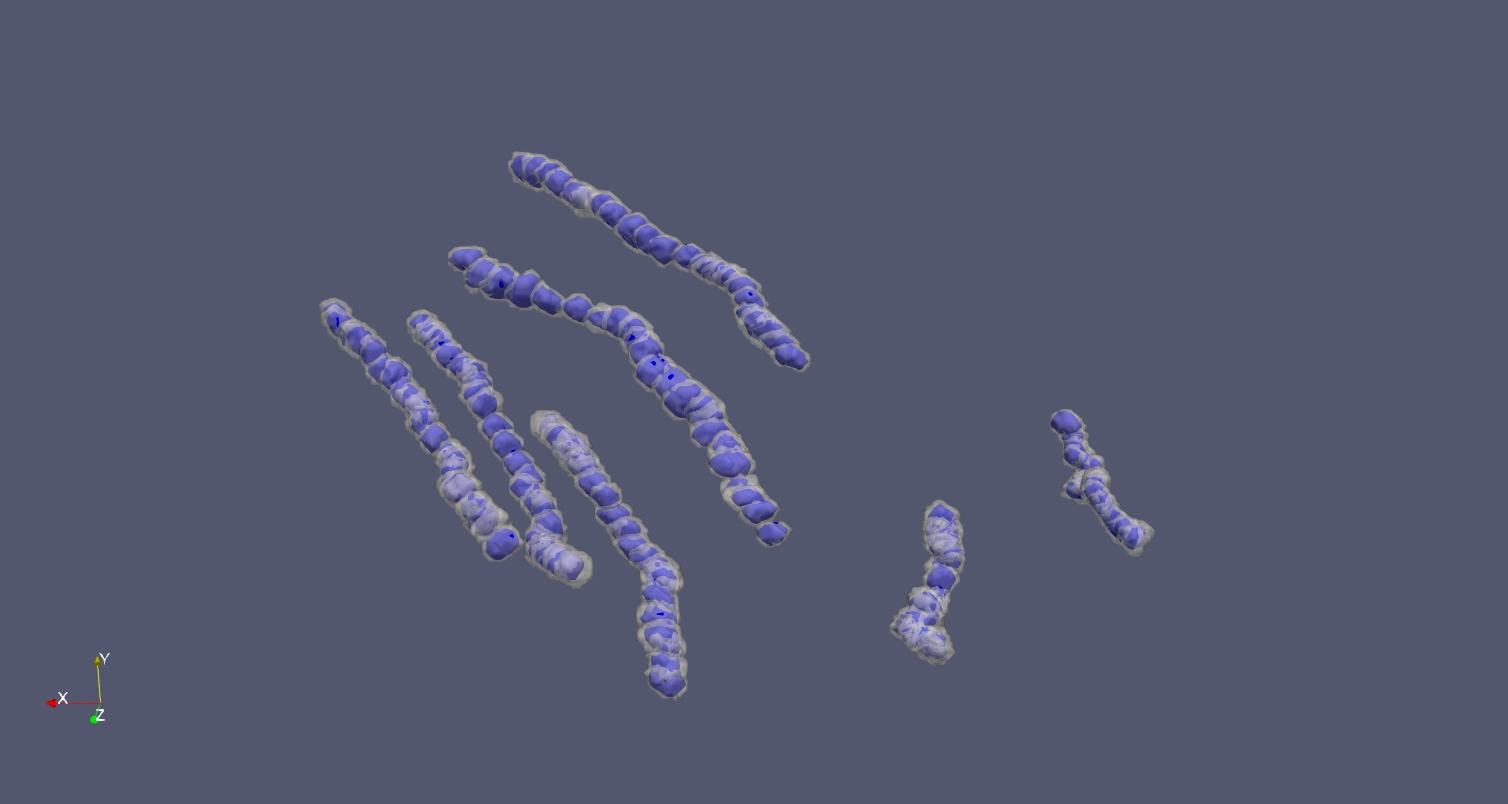}
	\end{minipage}
	\vspace{0.045cm}
	\caption{This figure shows the $3$D image of seven zebrafish cell nuclei moving in time (first image). Second image shows the segementation result using model \ref{first}.}
	\label{7_4d_inputimage}
\end{figure}

\noindent In the fourth experiment, we worked with original image intensity within the small $30 \times 30 \times 30$ computational domain. For each of the seven selected nuclei centers, the intensities around each nucleus are copied to the $30 \times 30 \times 30$ computational domain. Hence, there are seven small computational domains in each time frame or 3D volume, and each domain contains the original image intensity. The $3$D+time image processed is obtained by putting these original $3$D images (containing seven small 3D volumes) together. Additionally, we restricted computations to a small 3D volume around the nuclei. The images presented in Figure \ref{7_4d_orig_inputimage} show a visualization of the 3D+time images.
\noindent First image of Figure \ref{7_4d_orig_inputimage} shows the $3$D image of seven zebrafish cell nuclei that are moving in time and second image shows the corresponding segmentation result using the $4$D model (\ref{first}). The results of the segmentation are colored in black.
\begin{figure}[h]
\centering
	\begin{minipage}{0.44\textwidth}
		\centering
		\includegraphics[width=\textwidth]{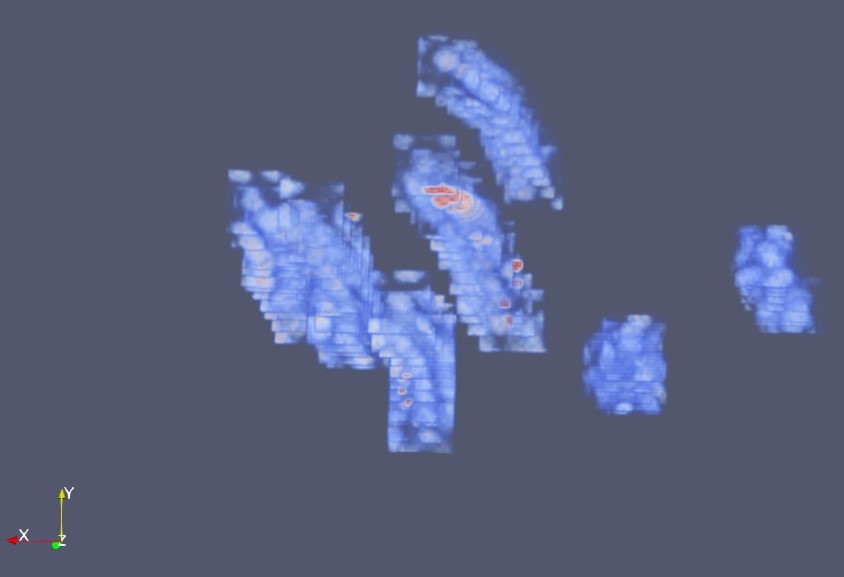}
	\end{minipage}
	\vspace{0.045cm}
	\begin{minipage}{0.44\textwidth}
		\centering
		\includegraphics[width=\textwidth]{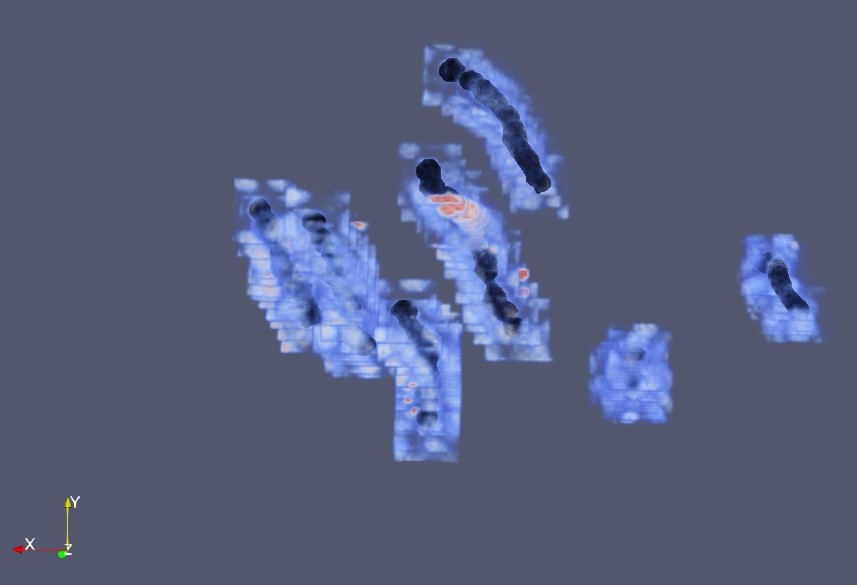}
	\end{minipage}
	\vspace{0.045cm}
	\caption{This figure shows the visualization of seven zebrafish cell nuclei moving in time.}
	\label{7_4d_orig_inputimage}
\end{figure}

\noindent Unlike in the previous four experiments where serial implementations were used, we will use the parallel implementation in subsequent experiments; thus, the entire 3D volumes of image intensities will be considered. We remark that the dataset used in the two subsequent experiments comprises several frames of 3D volumes. Each of these volumes has a dimension of $567\times577\times147$. Therefore, the maximum number of frames of the 3D volumes that can be processed at once in the cluster is 90. The reason for this maximum number is that the cluster has 1512 GB of memory available for computations. Additionally, to process the 90 frames of 3D volumes, with volume dimensions $567\times577\times147$, more than 1130 GB of memory is needed, and we decided
to limit the number of frames in our computations to a maximum of 70. Furthermore, to process these 70 frames of 3D volumes, with volume dimensions $567\times577\times147$, 1012 GB of memory was needed. This clearly shows that it may not be possible to process these images on a serial machine without parallel implementation utilizing the MPI.\\

\noindent In the fifth experiment (see \url{https://doi.org/10.5281/zenodo.5513118} for the videos of this experiment), we selected seven cell nuclei in 70 time frames within the zebrafish pectoral fin. The selected cell nuclei were clearly visible in all time frames. Thus, easy visualization of segmentation results in all time frames is the motivation for this selection. Additionally, in the two previous experiments of this section, we used an arbitrary number ``seven;'' hence, we have used seven in this experiment too.\\

\noindent In Figure \ref{orig_image_at_t040-070}, the first row shows the $3$D frames of the video at the time steps $1$ (first column), $20$ (second column), $40$ (third column), and $70$ (fourth column) and the second row shows the corresponding reconstruction of the selected seven nuclei images using (\ref{first}). Figure \ref{orig_image_at_t040-070} (see also, \url{https://doi.org/10.5281/zenodo.5513118}) shows that in each time frame, the seven selected cell nuclei, located nearly at the bottom of each image in the figure, were accurately reconstructed using the $4$D method. The segmentation results of these cell nuclei are colored in black at the bottom of each image in the second row of this figure.

\begin{figure}[h]
	\begin{minipage}{0.24\textwidth}
		\centering
		\includegraphics[width=\textwidth]{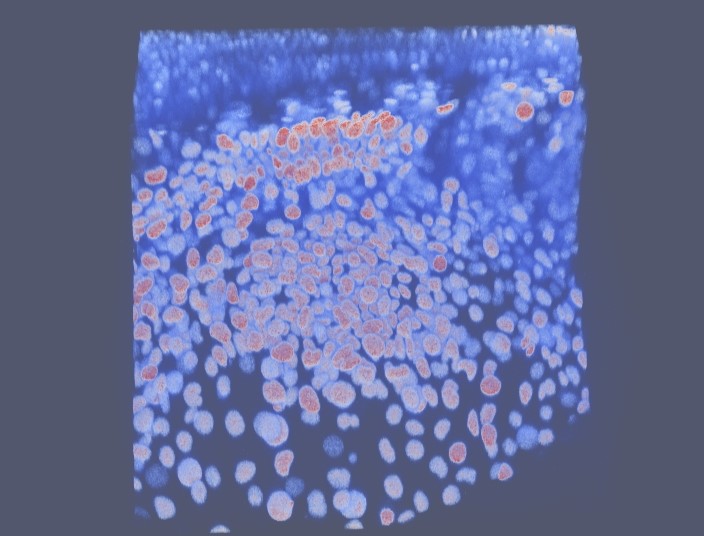}
	\end{minipage}
	\vspace{0.045cm}
	\begin{minipage}{0.24\textwidth}
		\centering
		\includegraphics[width=\textwidth]{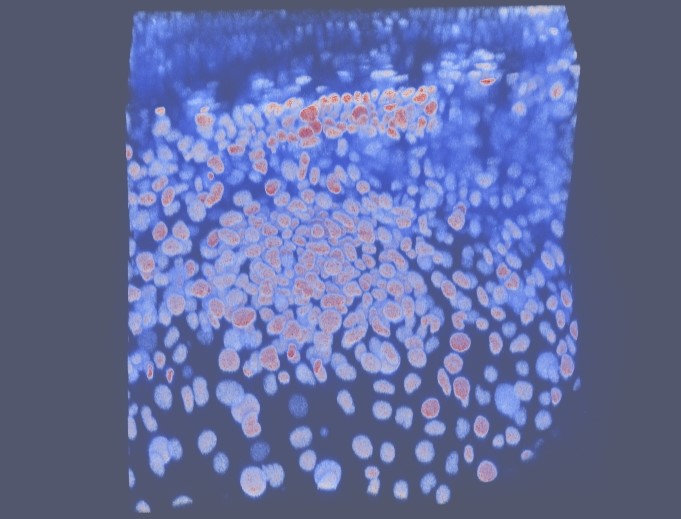}
	\end{minipage}
	\vspace{0.045cm}
	\begin{minipage}{0.24\textwidth}
		\centering
		\includegraphics[width=\textwidth]{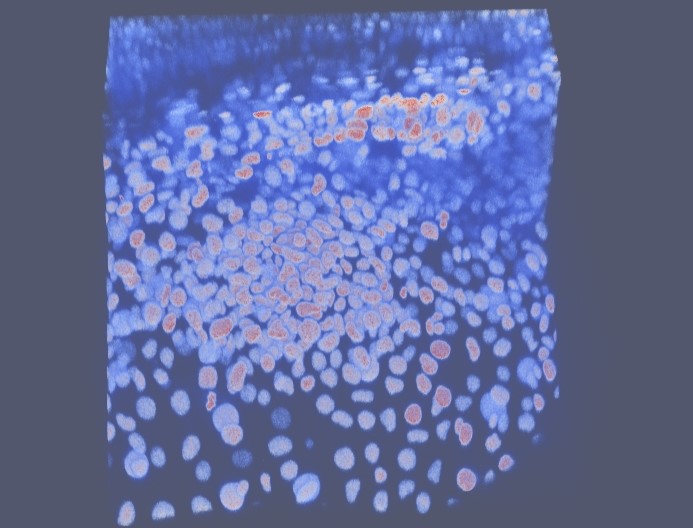}
	\end{minipage}
	\vspace{0.045cm}
	\begin{minipage}{0.24\textwidth}
		\centering
		\includegraphics[width=\textwidth]{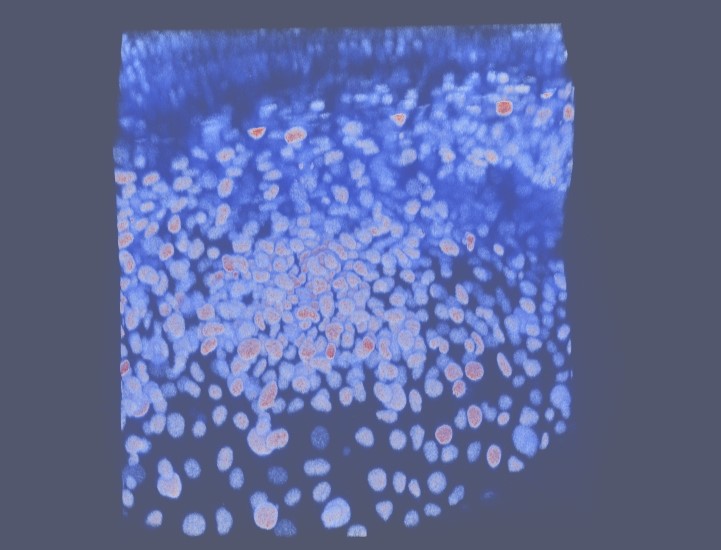}
	\end{minipage}
	\vspace{0.045cm}
	
	\begin{minipage}{0.24\textwidth}
		\centering
		\includegraphics[width=\textwidth]{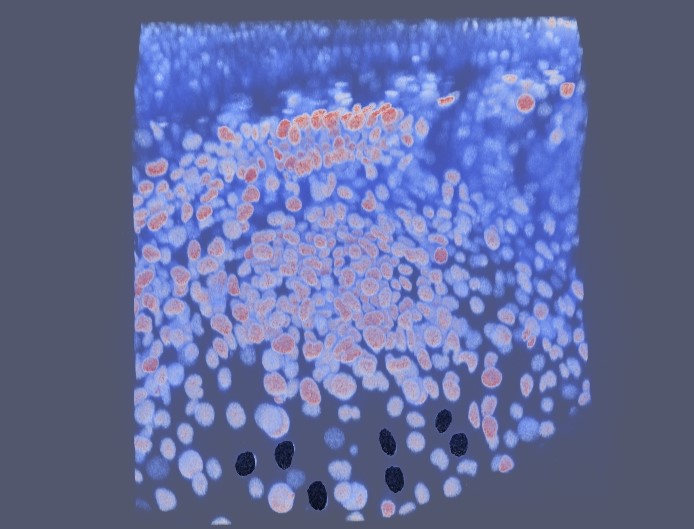}
	\end{minipage}
	\vspace{0.045cm}
	\begin{minipage}{0.24\textwidth}
		\centering
		\includegraphics[width=\textwidth]{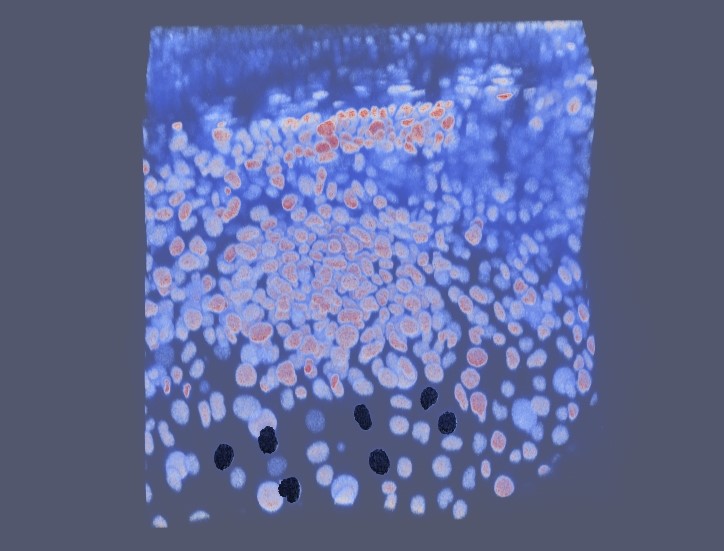}
	\end{minipage}
	\vspace{0.045cm}
	\begin{minipage}{0.24\textwidth}
		\centering
		\includegraphics[width=\textwidth]{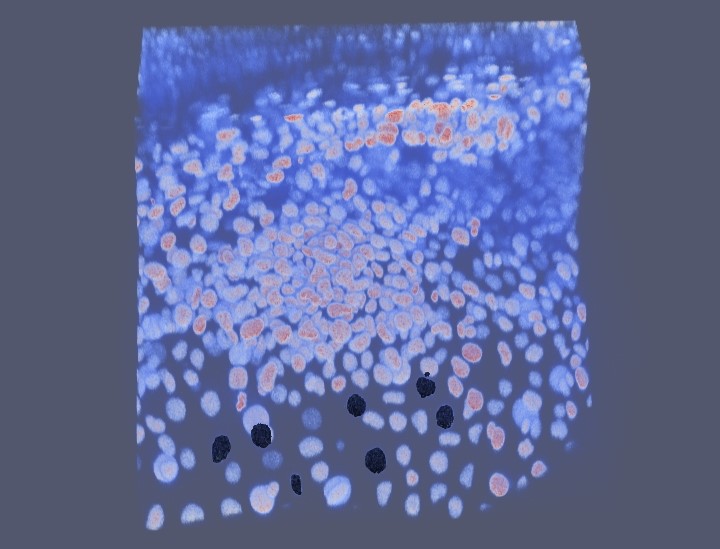}
	\end{minipage}
	\vspace{0.045cm}
	\begin{minipage}{0.24\textwidth}
		\centering
		\includegraphics[width=\textwidth]{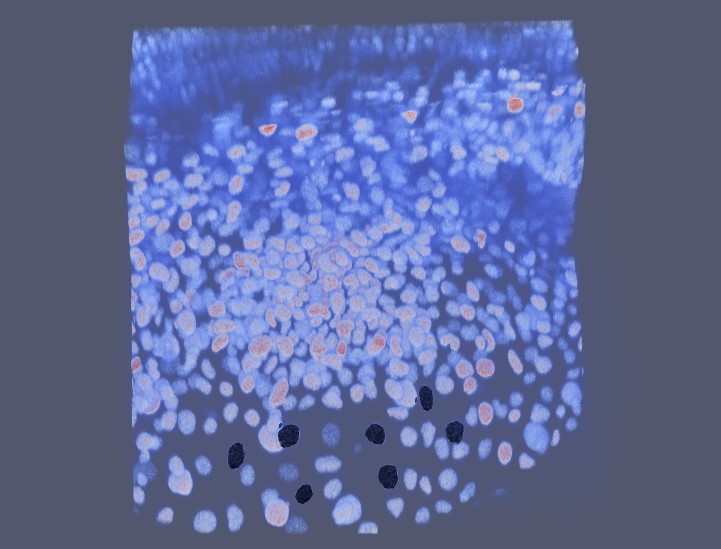}
	\end{minipage}
	\vspace{0.045cm}

	\caption{First row of this figure shows the $3$D frames of 3D+time microscopy images of  cell nuclei within the zebrafish pectoral fin at the time steps $1$ (first column), $20$ (second column), $40$ (third column), and $70$ (second column). The second row shows the corresponding reconstruction of the seven nuclei images using (\ref{first}).}
	\label{orig_image_at_t040-070}
\end{figure}

\noindent In the last experiment (see \url{https://doi.org/10.5281/zenodo.5513118} for the videos of this experiment), we reconstructed a group of selected cell nuclei in 70 time frames within the zebrafish pectoral fin. This cell population was selected by the biologists who provided the dataset. In Figure \ref{all_orig_image_at_t040-070}, the first row shows the $3$D frames of the video at the time steps $1$ (first column), $20$ (second column), $40$ (third column), and $70$ (fourth column) and the second row shows the corresponding reconstruction of the selected nuclei images using the $4$D method (\ref{first}). The segmentation results of this group of cell nuclei are colored in black at the center of each image in the second row of this figure. Furthermore, the selected group of cells are located inside the fin tissue and are covered by other cells nearer to the fin's surface. Consequently, it is not easy to visualize the segmentation results together with the original image intensity. This is why we selected the seven cell nuclei that are easily visualizable in the previous experiment. Moreover, Figure \ref{rttrack040-070} in Section \ref{celltracking} shows segmentation results of this experiment, which were used for cell tracking. Thus, we can deduce from these figures that the segmentation results appear to be correct.\\  

\noindent Throughout the experiments, the following recommended parameters were chosen automatically: $\delta = 0.5,$ $\vartheta = 0.5,$ $h = 0.01,$ $\tau = 0.01,$ $\varepsilon = 1,$ $\sigma = \frac{h^2}{8}$. The parameters $R = 12,$ $\alpha = 0.87,$ and $\beta = 0.13,$ are empirically chosen based on visual inspection of results or influence on the accuracy of tracking. Finally, $\theta$ is chosen based on available computational resources. For instance, in the experiments whose results are presented in Figures \ref{orig_image_at_t040-070} and \ref{all_orig_image_at_t040-070}, $\theta$ is $70$ and the maximum possible value for $\theta$ is $90$ due to memory availability.

\begin{figure}[h]
    \begin{minipage}{0.24\textwidth}
		\centering
		\includegraphics[width=\textwidth]{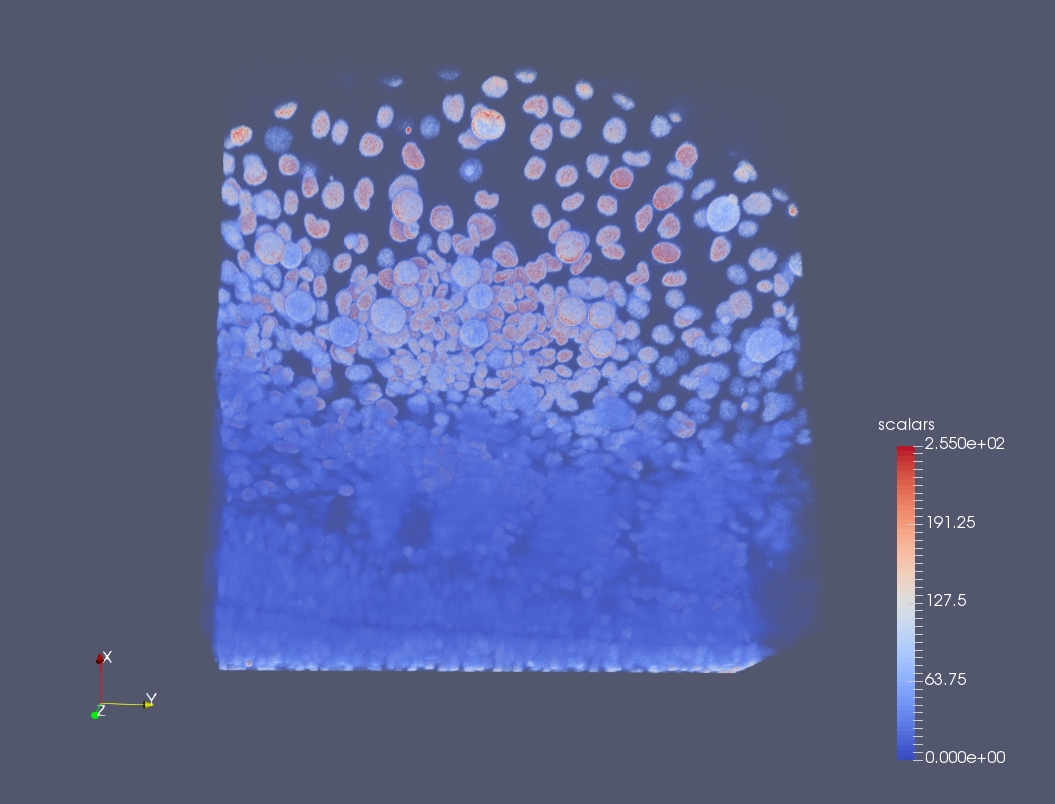}
	\end{minipage}
	\vspace{0.045cm}
	\begin{minipage}{0.24\textwidth}
		\centering
		\includegraphics[width=\textwidth]{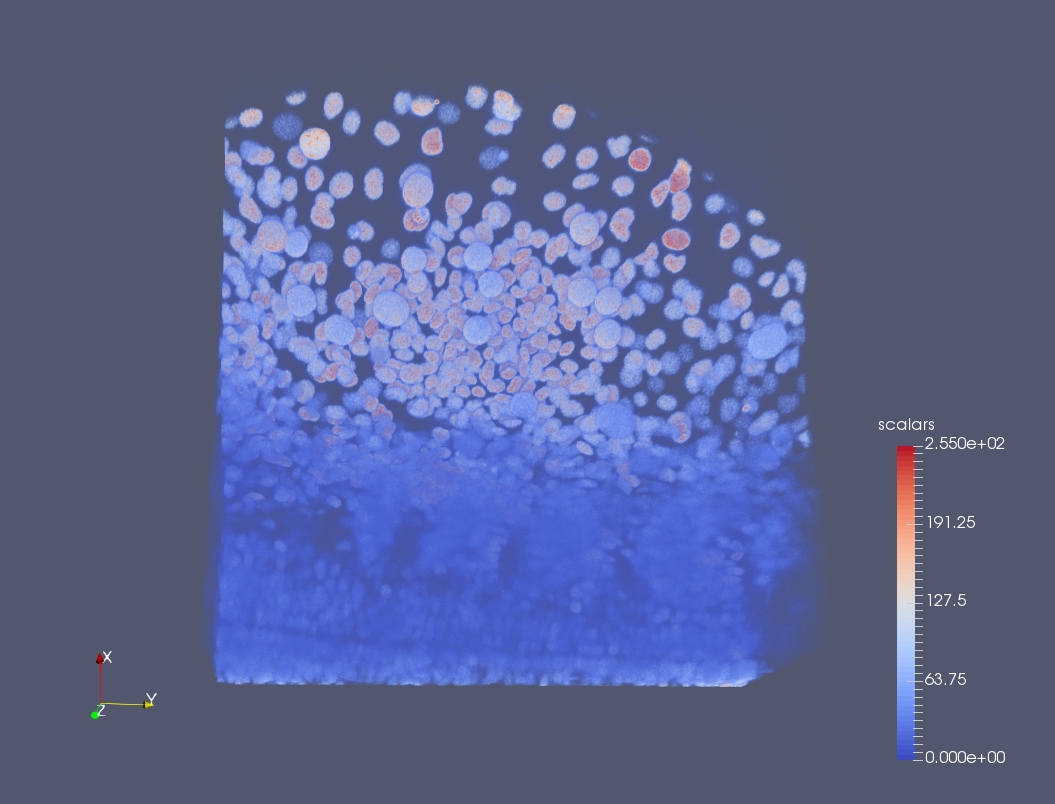}
	\end{minipage}
	\vspace{0.045cm}
	\begin{minipage}{0.24\textwidth}
		\centering
		\includegraphics[width=\textwidth]{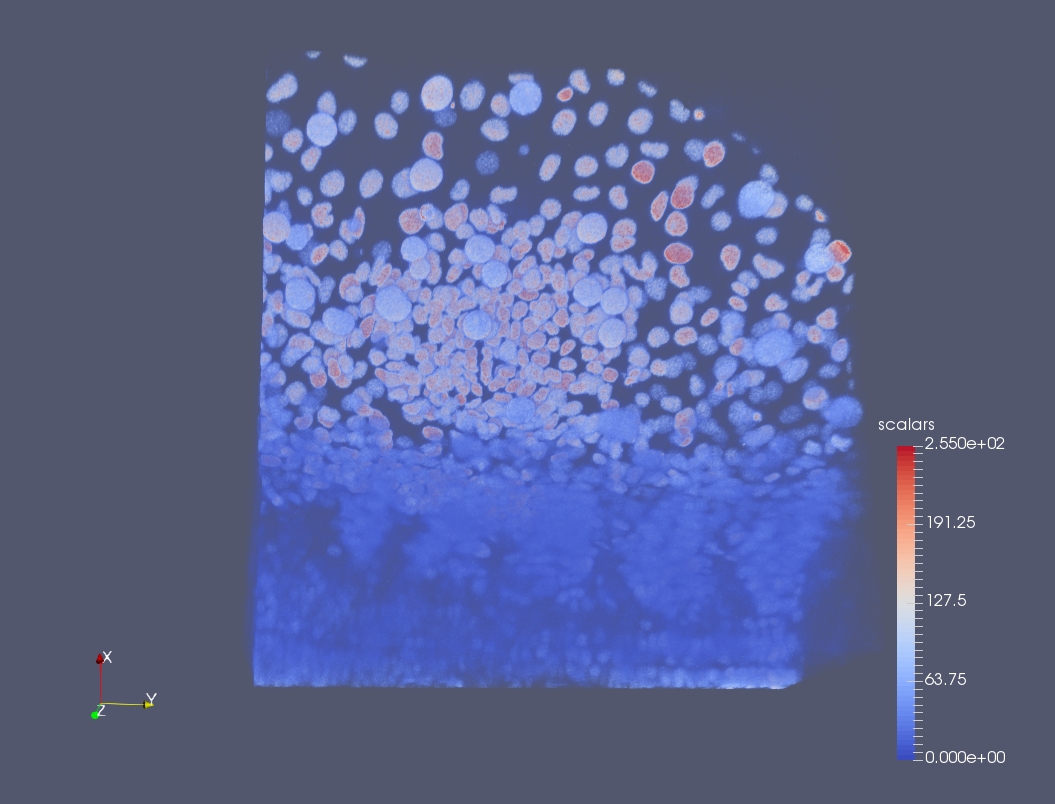}
	\end{minipage}
	\vspace{0.045cm}
	\begin{minipage}{0.24\textwidth}
		\centering
		\includegraphics[width=\textwidth]{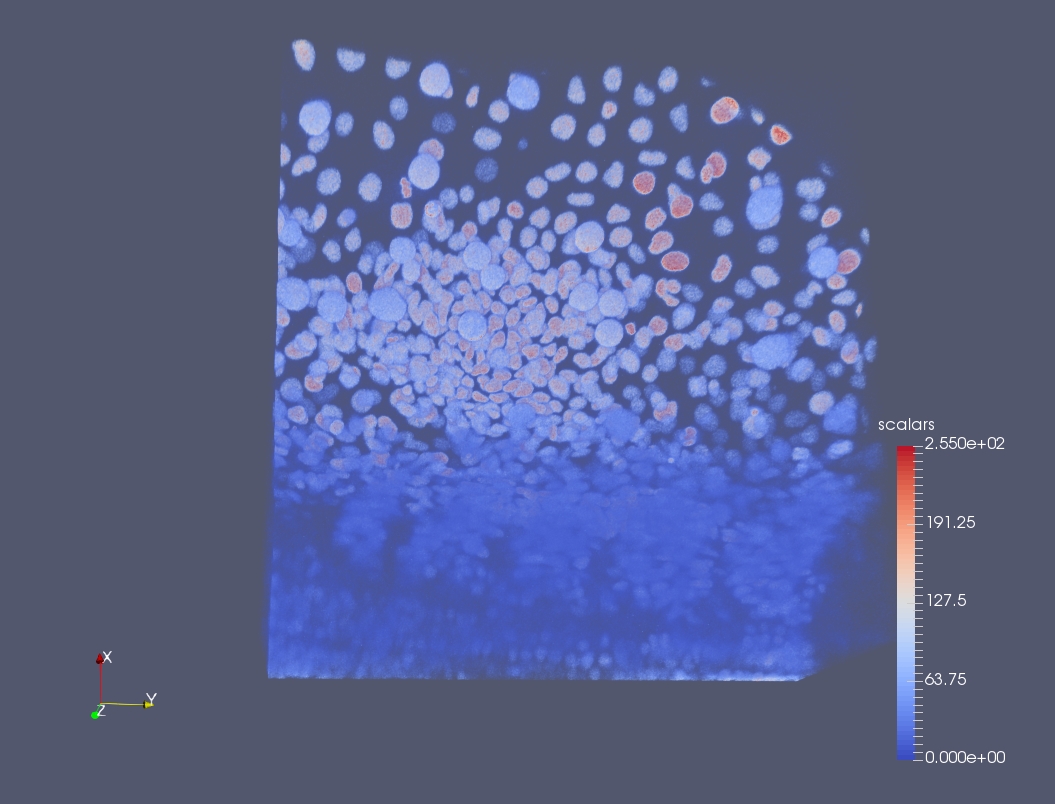}
	\end{minipage}
	\vspace{0.045cm}
	
	\begin{minipage}{0.24\textwidth}
		\centering
		\includegraphics[width=\textwidth]{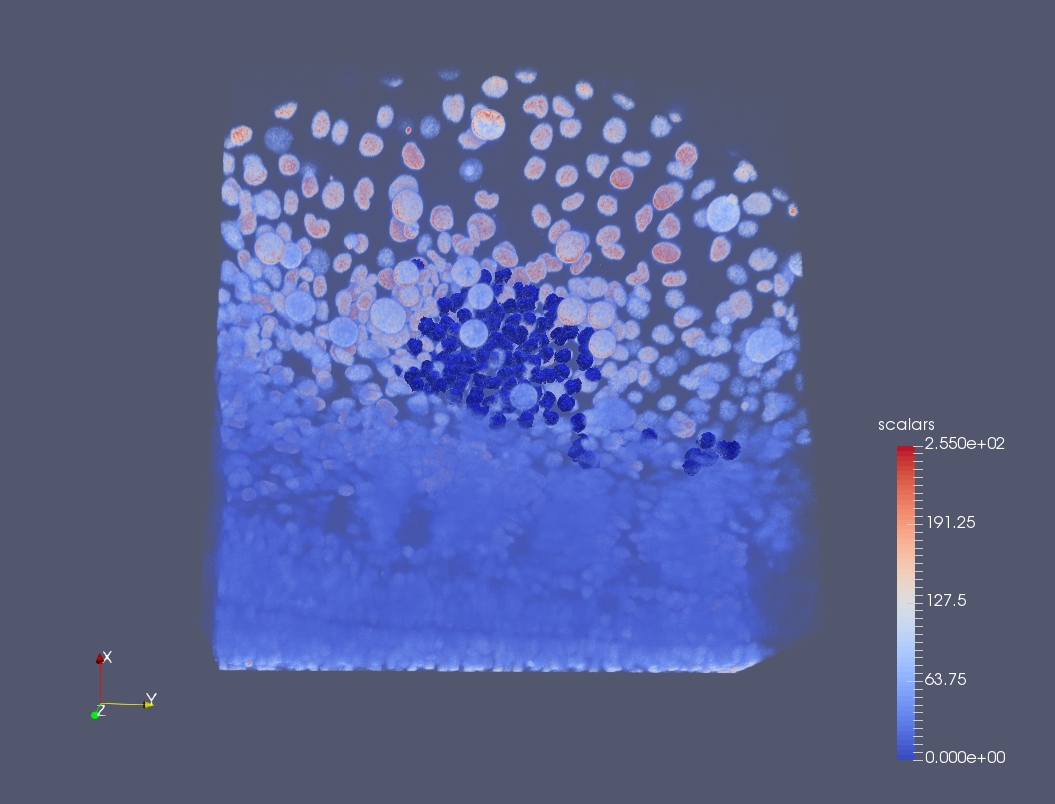}
	\end{minipage}
	\vspace{0.045cm}
	\begin{minipage}{0.24\textwidth}
		\centering
		\includegraphics[width=\textwidth]{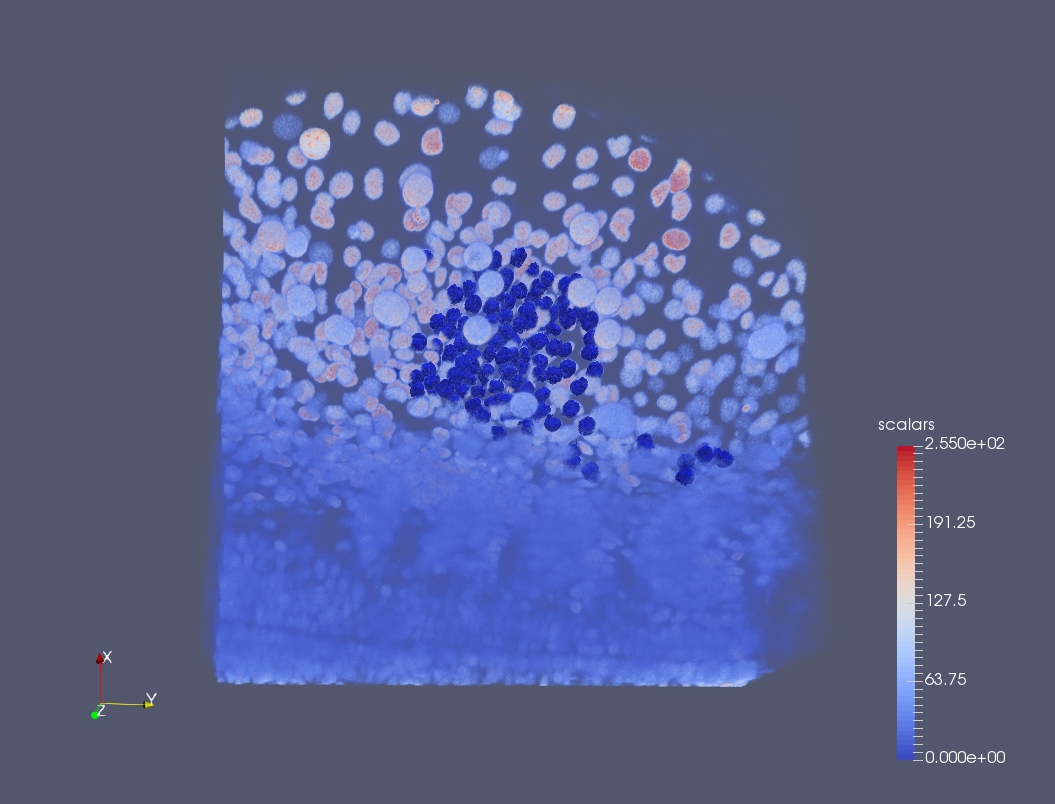}
	\end{minipage}
	\vspace{0.045cm}
	\begin{minipage}{0.24\textwidth}
		\centering
		\includegraphics[width=\textwidth]{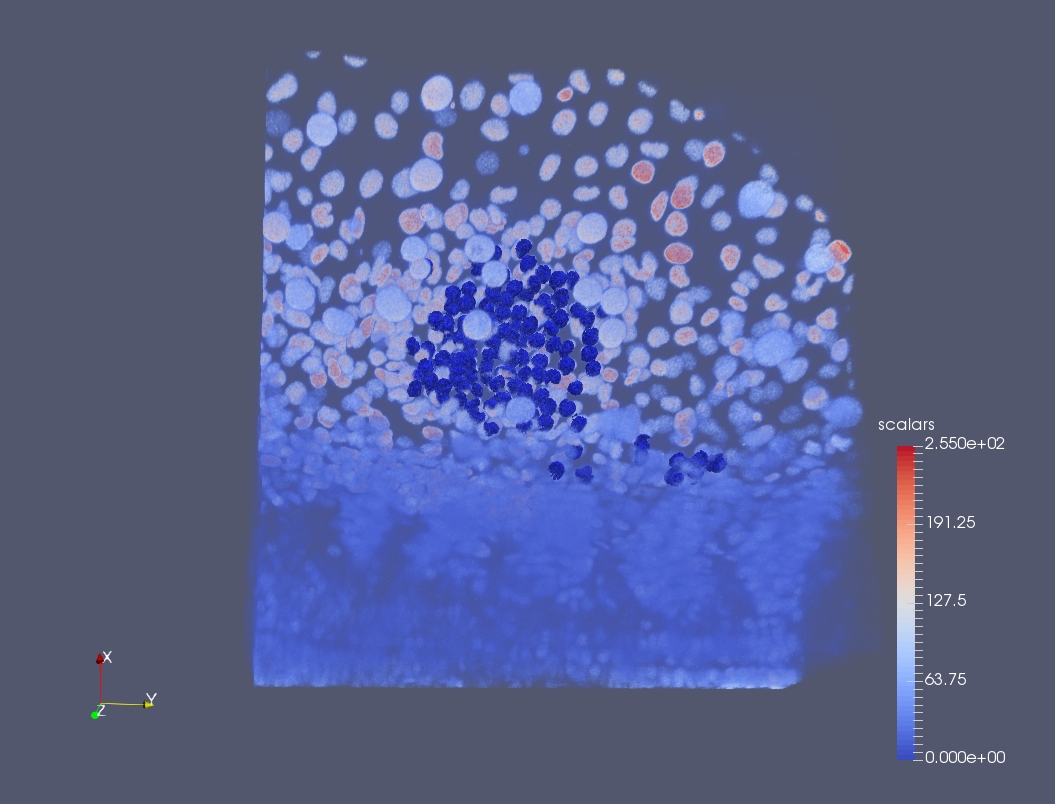}
	\end{minipage}
	\vspace{0.045cm}
	\begin{minipage}{0.24\textwidth}
		\centering
		\includegraphics[width=\textwidth]{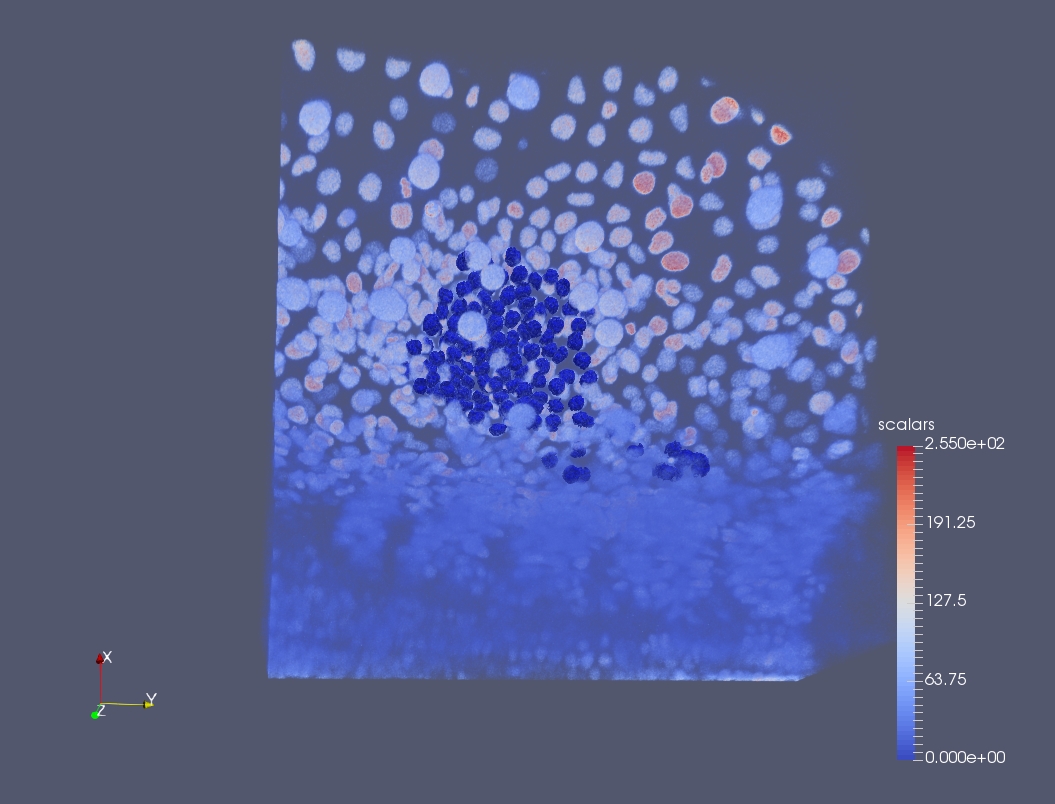}
	\end{minipage}
	\vspace{0.045cm}
	
	\caption{First row of this figure shows the $3$D frames of 3D+time microscopy images of  cell nuclei within the zebrafish pectoral fin at the time steps $1$ (first column), $20$ (second column), $40$ (third column), and $70$ (fourth column) and the second row shows the corresponding reconstruction of the nuclei images using (\ref{first}).}
	\label{all_orig_image_at_t040-070}
\end{figure}

\section{Cell tracking based on 4D segmentation}\label{celltracking}
\noindent In this section, the results of cell tracking based on the 4D segmentation method are presented. The cell tracking procedure is briefly described below, and details of the algorithm are given in \cite{park}; for previous works on cell tracking, see \cite{4dcelltracking1} and \cite{4dcelltracking2}.\\

\noindent After we obtain the 4D segmentation of 3D+time video, first, we label all voxels in all time frames belonging to any 3D simply connected segmented region; we refer to such 3D simply connected segmented region as “segmented cell.” The labeling distinguishes between the inner cell volumes and the “background” corresponding to the outside of the cells’ region. Additionally, inside the segmented cells, we compute a 3D distance from the segmented cell boundary. By detecting the voxel in which the maximal distance is achieved, we obtain an approximate cell center. The goal of tracking is to interconnect in time such approximate cell centers obtained for every cell in every time frame. We use an iterative backtracking approach to that goal, starting from the last time frame and going backward in time through all simply connected regions of 4D segmentation.  First, we project the voxel representing the cell center detected in the current time frame to spatially same voxel in the previous time frame. If the projected voxel belongs to some cell in the previous time frame, we find its center and connect two cell centers forming one section of the cell trajectory. If the projected voxel is outside of any segmented cell in the previous time frame, all voxels of the cell at the current time frame are inspected and checked whether its projection to the previous time frame is inside of any simply connected segmented region. If there is an overlap between the cells at the current and the previous time frame, we find the cell center of the overlapping region at the previous time frame and construct the section of partial trajectory again. If no overlap is found, the partial trajectory ends at the current time frame and does not continue further backward. After constructing all partial trajectories inside the 4D segmentation by the approach described above, we interconnect them if a suitable condition is fulfilled. We estimate the tangent of the partial trajectory in its first and last point and prolong it to the corresponding time frames. If the last/first point of another trajectory exists in the previous/next time frame in a closed neighborhood of the prolonged trajectory, we interconnect them. As we can see, the tracking procedure is quite simple, mainly due to the utilization of the 4D segmentation. 

\subsection{Numerical experiments} \label{num_tube}
\noindent In the first experiment (see \url{https://doi.org/10.5281/zenodo.5513089} for the complete video), results of segmentation by the 4D method (\ref{first}) were used as a basis to track the artificially generated spheres shown in Figure \ref{artificial_image_at_t1-20}. In Figure \ref{atrack}, $3$D frames of the sphere trajectories (in this case, straight lines) at the time frames $1$, $10$, and $20$ were shown. In this artificial dataset, we may recall that there are four spheres in time frames $1−3$, seven spheres in time frames $4−17$, and five spheres in time frames $18−20$. So, this explains why there are four, seven, and five spheres in this figure. In each of the three pictures, the straight lines show the trajectories of the spheres, and the square at the end of each line or trajectory shows the sphere's present time. The trajectories are correct because, from the construction of the dataset, all spheres move in straight lines. We note that the viewpoint of the images presented in this figure is different from the viewpoint of the images in Figure \ref{artificial_image_at_t1-20}. In presenting the result of this experiment, we choose this viewpoint because it has the best view of all trajectories.

\begin{figure}[h]
\centering
	\begin{minipage}{0.24\textwidth}
		\centering
		\includegraphics[width=\textwidth]{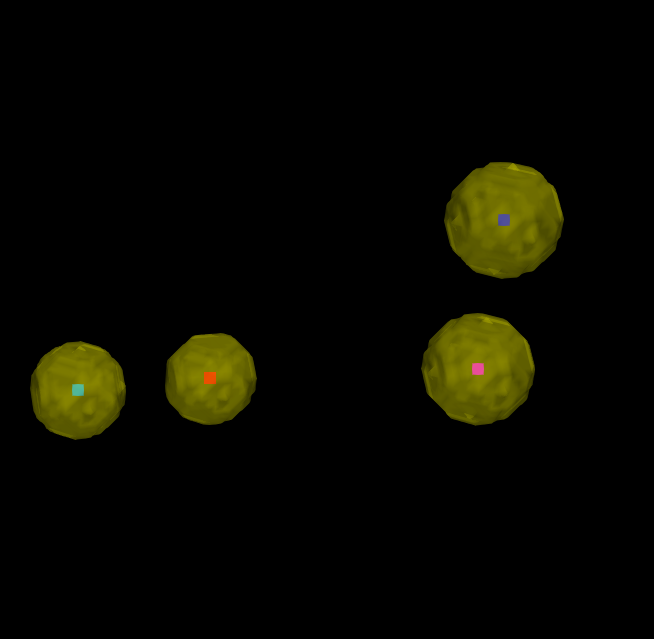}
	\end{minipage}
	\vspace{0.045cm}
	\begin{minipage}{0.24\textwidth}
		\centering
		\includegraphics[width=\textwidth]{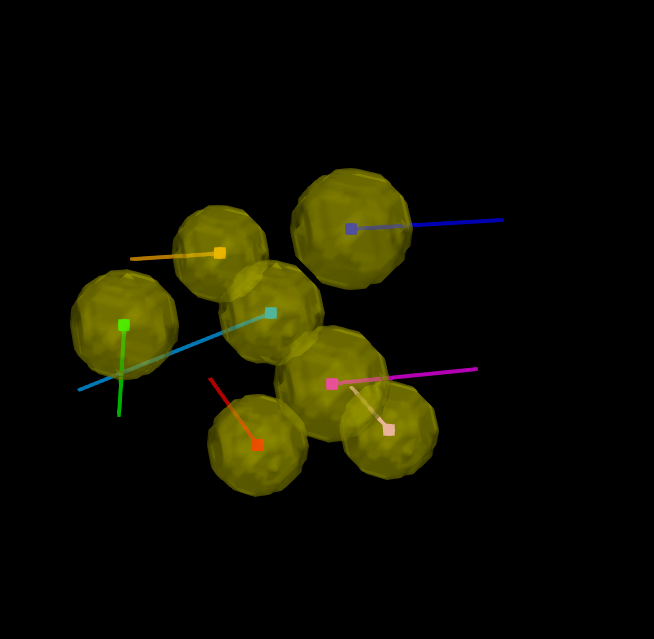}
	\end{minipage}
	\vspace{0.045cm}
	\begin{minipage}{0.24\textwidth}
		\centering
		\includegraphics[width=\textwidth]{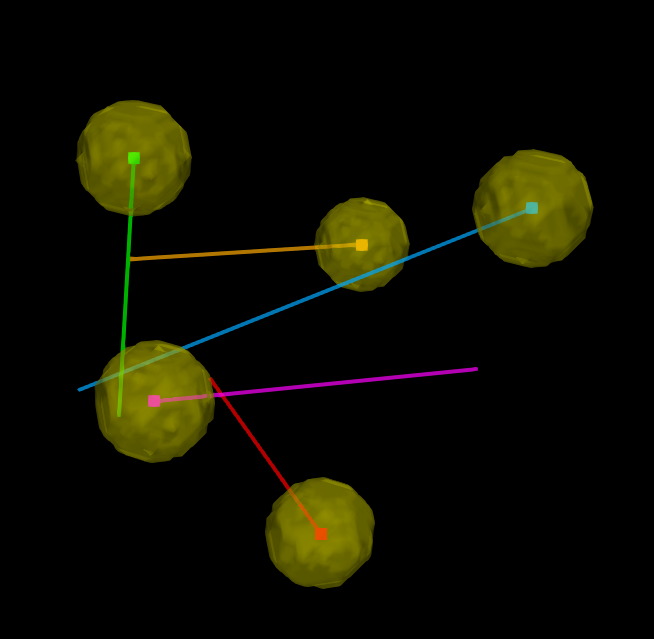}
	\end{minipage}
	\vspace{0.045cm}
	
	\caption{This figure shows $3$D frames of the sphere trajectories at time frames $1$ (first column), $10$ (second column), and $20$ (third column).}
	\label{atrack}
\end{figure}
 
\noindent In the second experiment (see \url{https://doi.org/10.5281/zenodo.5513118} for the complete video), results of segmentation by the 4D method were also used as a basis to track the 3D+time microscopy images of selected seven cell  nuclei  within  the  zebrafish  pectoral  fin shown in Figure \ref{orig_image_at_t040-070}. In Figure \ref{rtrack040-070}, $3$D frames of the trajectories of 3D+time microscopy images  of seven cell  nuclei  within  the  zebrafish  pectoral  fin  at  the time  steps $1$, $20$, $40$, and $70$ were shown. The trajectories of these tracked cells show the movement of the cells over time. Additionally, from the last frame (second column) in Figure \ref{rtrack040-070}, we see that these cells are not moving in straight lines. Instead, they move in an arbitrary random direction.\\

\begin{figure}[h]
\centering
	\begin{minipage}{0.44\textwidth}
		\centering
		\includegraphics[width=\textwidth]{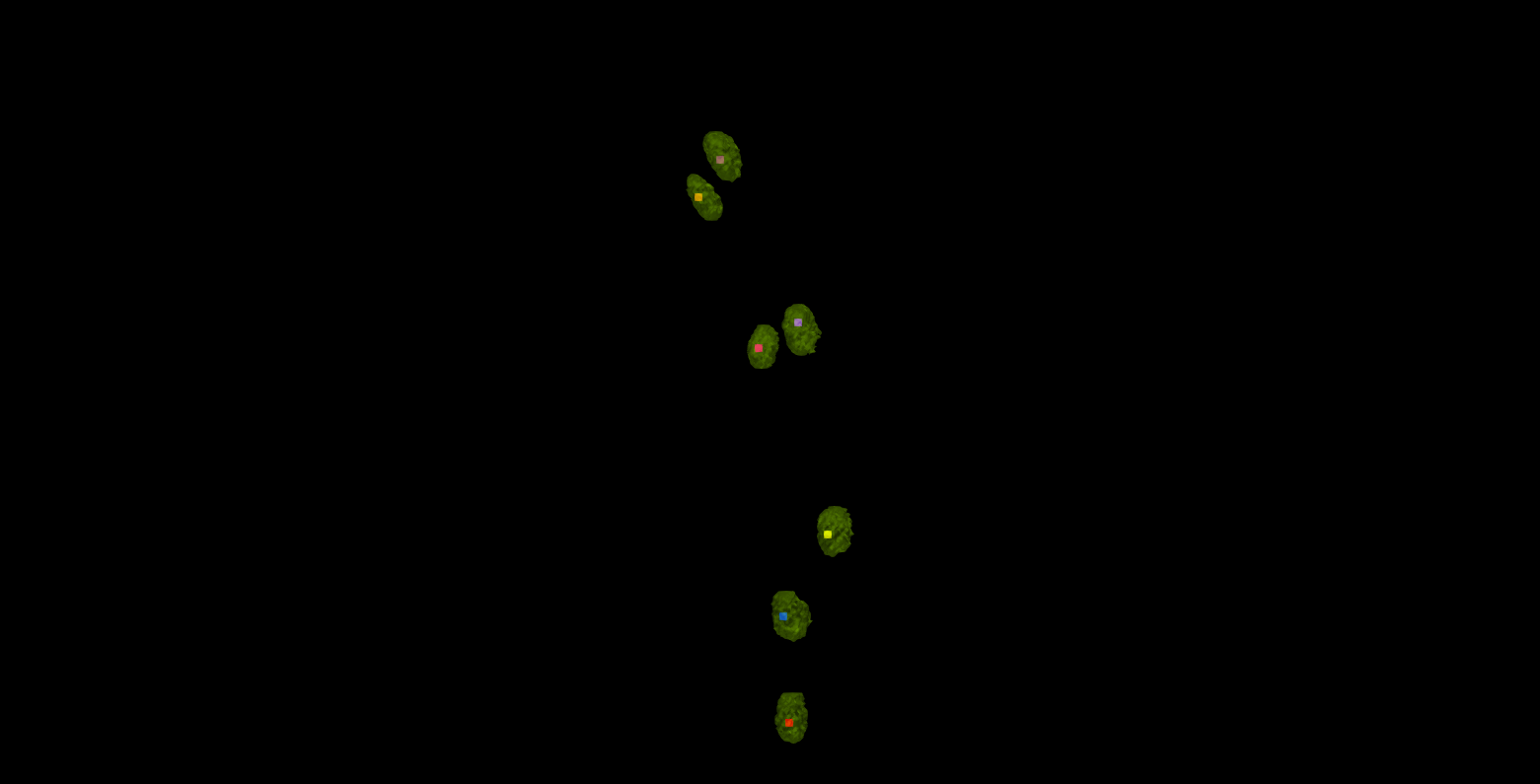}
	\end{minipage}
	\vspace{0.045cm}
	\begin{minipage}{0.44\textwidth}
		\centering
		\includegraphics[width=\textwidth]{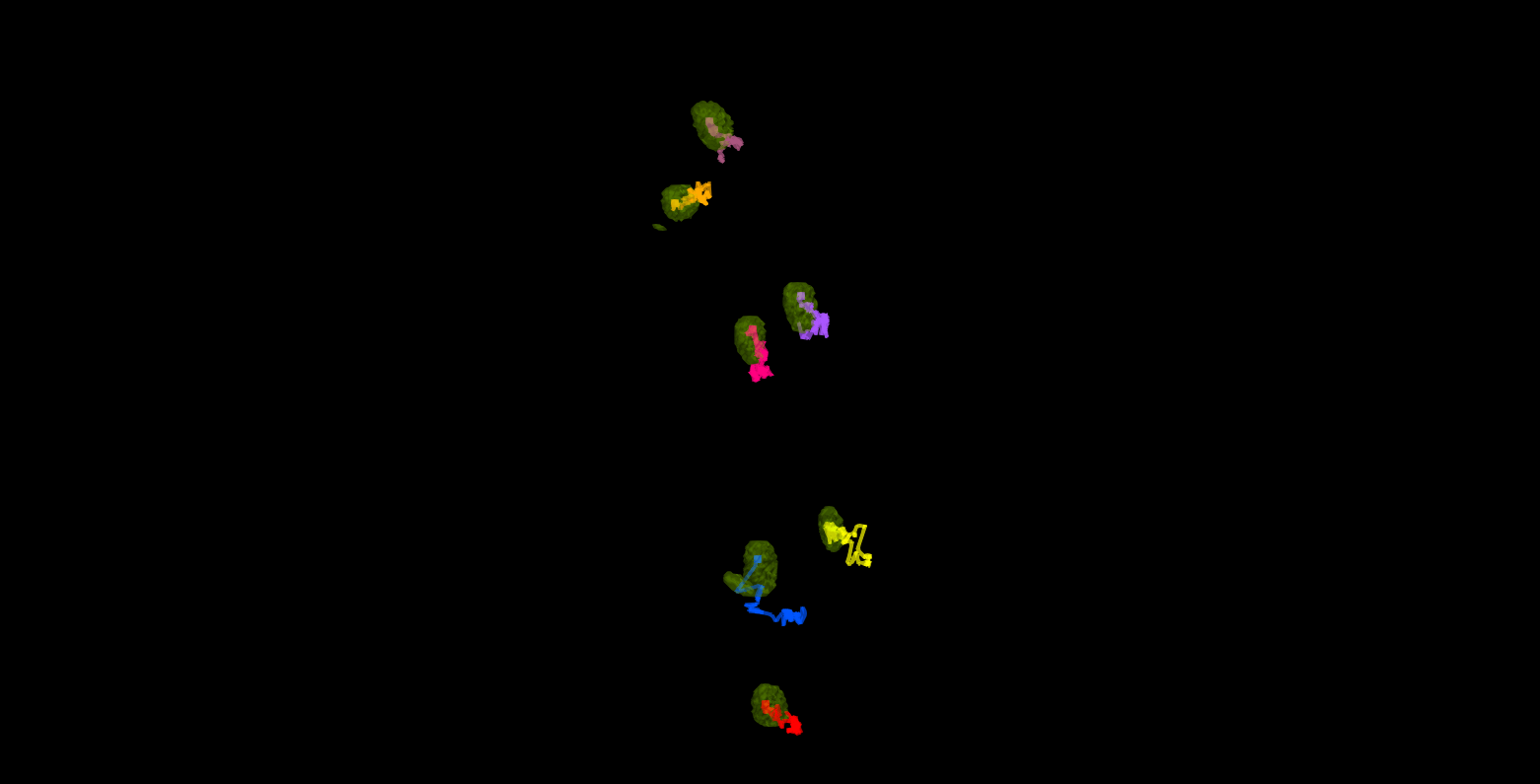}
	\end{minipage}
	\vspace{0.045cm}
	
	\caption{This figure shows $3$D frames of the trajectories of $3$D+time microscopy images of seven cell nuclei within the zebrafish pectoral fin at time frames $1$ (first column) and $70$ (second column).}
	\label{rtrack040-070}
\end{figure}

\noindent In the third experiment (see \url{https://doi.org/10.5281/zenodo.5513118} for the complete video), results of segmentation by our 4D method were again used as a basis to track the 3D+time microscopy images  of a group of selected cell nuclei in 70 time frames within  the  zebrafish  pectoral  fin shown in Figure \ref{all_orig_image_at_t040-070}. In Figure \ref{rttrack040-070}, $3$D frames of the trajectories of 3D+time images of five cell nuclei selected from the original population within the zebrafish pectoral fin at the time frames $1$, $20$, $40$, and $70$ were presented. The five selected cell trajectories are the ones that showed significant cell movement in the selected population. Additionally, the trajectories of these five selected cells are colored yellow, blue, red, cyan, and pink. Furthermore, in the second image of Figure \ref{rttrack040-070}, the initial positions of the nuclei were colored gray, while the current positions are colored green.

\begin{figure}[h]
\centering
	\begin{minipage}{0.44\textwidth}
		\centering
		\includegraphics[width=\textwidth]{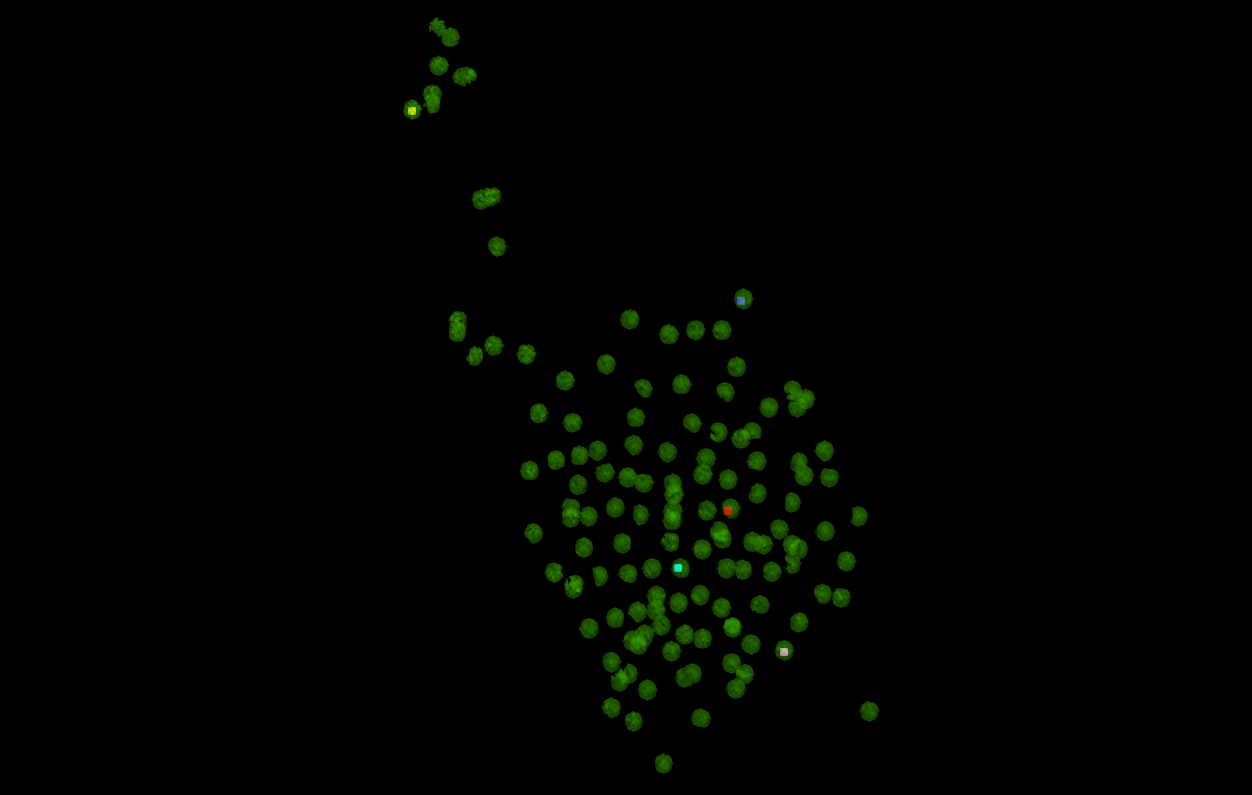}
	\end{minipage}
	\vspace{0.045cm}
	\begin{minipage}{0.44\textwidth}
		\centering
		\includegraphics[width=\textwidth]{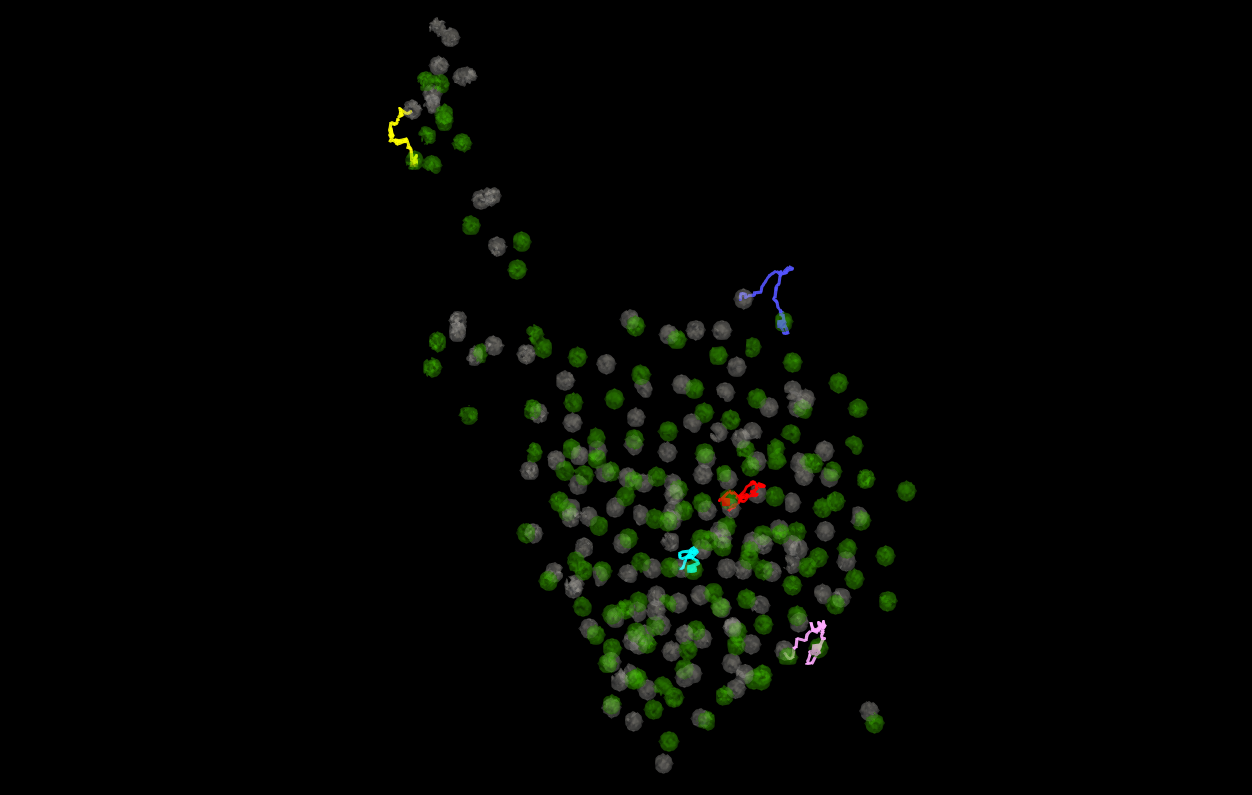}
	\end{minipage}
	\vspace{0.045cm}
	
	\caption{This figure shows $3$D frames of the trajectories of 3D+time images of five cell nuclei within the zebrafish pectoral fin at time frames $1$ (first row) and $70$ (second row).}
	\label{rttrack040-070}
\end{figure}

\noindent Finally, from the results presented in Figures \ref{atrack}, \ref{rtrack040-070}, and \ref{rttrack040-070}, we can conclude that our 4D method (\ref{first}) serves as basis for cell tracking.

\section{Conclusion}
\noindent We suggested a segmentation approach that is based on the surface evolution governed by a nonlinear PDE. The approach is a generalization of the SUBSURF model \cite{subsurf}. This generalization is achieved by defining the input to the edge detector function as the weighted sum of norm of presmoothed 4D image and norm of presmoothed thresholded 4D image in a local neighborhood. We introduced and studied a numerical method for the solution of the studied model. The numerical method is based on a finite volume approach and used the reduced diamond cell approach to approximate the gradient of the solution. For numerical discretization, we used a semi-implicit finite volume scheme. We proved that the numerical scheme employed is unconditionally stable. As a special case in the 3D framework, we applied the studied 4D method to 3D image segmentation of cell nuclei and membranes within the pectoral fin of developing zebrafish embryo and cell nuclei in developing mouse embryo. Furthermore, we employed the OpenMP and MPI for the studied model's efficient (parallel) computer implementation. Several numerical examples were presented to demonstrate how the model works both for 3D and 4D images. Moreover, we included the cell tracking results to show how our new method serves as a basis for tracking. All parallel computations were done on a Linux cluster on a server comprising six computational nodes. Each computational node has $32$ processors and 
$252$ GB of memory; thus, the cluster has 192 processors and 1512 GB of memory available for computations. In the real application presented in the fifth and sixth experiments of Section \ref{chapterrealapp}, to process the 3D+time microscopy images with dimensions $567\times577\times147\times70$, 1012 GB of memory was needed. This clearly shows that it may not be possible to process these images on a serial machine without parallel implementation utilizing the MPI. Finally, from the results presented, we conclude that the mathematical model (\ref{first}) is a useful and successful generalization of the classical SUBSURF model.

\begin{acknowledgements}
This work received funding from the European Union's Horizon 2020 research and innovation programme under the Marie Sk\l{}odowska-Curie grant agreement No. 721537, and the projects APVV-19-0460 and VEGA
1/0436/20.
\end{acknowledgements}
\section*{Conflict of interest}
 The authors declare that they have no conflict of interest.

\renewcommand{\bibname}{References}



\end{document}